\documentclass[12pt,reqno]{amsart}  

\usepackage{amsmath,amsthm,amssymb,amscd,mathdots,mathtools,enumerate,shuffle,stmaryrd}
\usepackage{mathabx} % for \widecheck
\usepackage{hyperref}
\usepackage{tikz,tikz-cd}
\allowdisplaybreaks
\usetikzlibrary{decorations.pathmorphing,shapes,calc}

\newtheorem{theorem}{Theorem}[section]
\newtheorem*{theorem*}{Theorem}
\newtheorem{proposition}[theorem]{Proposition}

\newtheorem{lemma}[theorem]{Lemma}
\newtheorem{corollary}[theorem]{Corollary}

\numberwithin{equation}{section}

\theoremstyle{definition}
\newtheorem{definition}[theorem]{Definition}
\newtheorem{example}[theorem]{Example}
\newtheorem{remark}[theorem]{Remark}
\mathtoolsset{showonlyrefs} % just show used refs

\newcommand{\h}{\mathfrak H}

\newcommand{\Ha}{\mathbb{H}}
\newcommand{\Q}{\mathbb{Q}}
\newcommand{\Z}{\mathbb{Z}}
\newcommand{\R}{\mathbb{R}}

\newcommand{\C}{\mathbb{C}}

\newcommand{\Zg}{Z_\gamma}

\newcommand{\one}{{\bf 1}}
\newcommand{\gdiamond}{\hat{\diamond}}
\newcommand{\qsh}{\ast_\diamond}
\newcommand{\gqsh}{\ast_{\gdiamond}}
\newcommand{\QL}{\Q\langle L \rangle}
\newcommand{\QLZ}{\Q\langle L_z \rangle}
\newcommand{\QLB}{\Q\langle \LB \rangle}

\newcommand{\LB}{L_z^{\rm{bi}}}
\newcommand{\LZ}{L_z}

\DeclareRobustCommand{\bi}{\genfrac{(}{)}{0pt}{}}

\newcommand{\gbg}{\mathfrak{G}}

\newcommand{\g}{\mathfrak{g}}
\newcommand{\bm}{\mathfrak{b}} % mould rat sol dsh
\newcommand{\bb}{\mathfrak{b}} % bimould rat sol dsh
\newcommand{\bR}{\tilde{\bb}}
\newcommand{\LL}{\mathfrak{L}}
\newcommand{\gil}{\g^{\ast}}

\newcommand{\A}{\mathcal{A}}
\newcommand{\mz}{\mathcal Z}
\newcommand{\CbMES}{\mathcal{G}^{{\rm bi}}}
\newcommand{\CMES}{\mathcal{G}}

\definecolor{mycolor}{RGB}{194, 8, 88}
\newcommand{\todo}[1]{\message{LaTeX Warning: You did not finish your work :-( on input line \the\inputlineno} {\color{mycolor} {\big[\,}{\bf Todo:} #1\,\big]}}

\title{Combinatorial multiple Eisenstein series}

\author{Henrik Bachmann}
\address{Graduate School of Mathematics, Nagoya University, Nagoya, Japan.}
\email{henrik.bachmann@math.nagoya-u.ac.jp}

\author{Annika Burmester}
\address{Faculty of Mathematics, Bielefeld University, Bielefeld, Germany.}
\email{aburmester@math.uni-bielefeld.de}

\subjclass[2020]{Primary 
11M32; %multizeta values
Secondary 11F11 %Holomorphic modular forms of integral weight
}
\keywords{Multiple zeta values, double shuffle equations, Eisenstein series, modular forms}

\begin{document}
\date{\today}
\maketitle

\begin{abstract} 
We construct a family of $q$-series with rational coefficients satisfying a variant of the extended double shuffle equations, which are a lift of a given $\Q$-valued solution of the extended double shuffle equations.  We call these $q$-series combinatorial (bi-)multiple Eisenstein series, and in depth one they coincide with (classical) Eisenstein series. Combinatorial multiple Eisenstein series can be seen as an interpolation between the given $\Q$-valued solution of the extended double shuffle equations (as $q\rightarrow 0$) and multiple zeta values (as $q\rightarrow 1$). In particular, they are $q$-analogues of multiple zeta values closely related to modular forms. Their definition is inspired by the Fourier expansion of multiple Eisenstein series introduced by Gangl-Kaneko-Zagier. Our explicit construction is done on the level of their generating series, which we show to be a so-called symmetril and swap invariant bimould.
\end{abstract}

\section{Introduction}
Multiple zeta values are defined for $r\geq 1$ and $k_1\geq 2, k_2,\dots,k_r \geq 1$ by
\begin{align}\label{eq:defmzv}
	\zeta(k_1,\ldots,k_r) =\sum_{m_1>\cdots>m_r>0} \frac{1}{m_1^{k_1}\cdots m_r^{k_r}}.
\end{align}
We call the number $k_1+\dots +k_r$ its \emph{weight} and $r$ its \emph{depth}. By $\mz$, we denote the $\Q$-algebra of all multiple zeta values. There are two ways of expressing the product of multiple zeta values, and both can be written in terms of quasi-shuffle products (\cite{H}). The relations obtained from the two product expressions, together with some regularization process, are referred to as the \emph{extended double shuffle relations} of multiple zeta values (\cite{IKZ}). Conjecturally these give all algebraic relations among multiple zeta values. Multiple zeta values have various different connections to modular forms. For example, in the case $r=1$ multiple zeta values are the Riemann zeta values, which also appear as the constant term of Eisenstein series. In \cite{GKZ} the authors defined double Eisenstein series, which have double zeta values (\eqref{eq:defmzv} in the case $r=2$) as their constant terms, and which in some sense give a natural depth two version of Eisenstein series. This raised the question if these objects also satisfy some of the extended double shuffle relations. Partial answers for this were given in \cite{GKZ} and in arbitrary depths for so-called multiple Eisenstein series in \cite{B1}, \cite{B2} and \cite{BT}. In this work, we present a new approach and lift Eisenstein series with rational coefficients in a purely combinatorial\footnote{In the sense that we work with formal $q$-series without any convergence issues in contrast to working with holomorphic functions given as sums over lattice points.} way to $q$-series which we call combinatorial multiple Eisenstein series. This provides a new framework for relating modular forms and multiple zeta values. We discuss the relations satisfied by these $q$-series and give an interpretation of them as a variant of the extended double shuffle relations. 

We first recall the extended double shuffle relations for multiple zeta values before explaining how the combinatorial multiple Eisenstein series fit into the picture. 
Consider the alphabet $L_z=\{z_k \mid k\geq 1\}$ and let $\h^1 = \QLZ$ be the free algebra over $L_z$. Define a product on $\Q L_z$ by $z_{i} \diamond z_{j} = z_{i+j}$ for all $i,j\geq 1$. The corresponding quasi-shuffle product \eqref{eq:qshdef} $\ast = \qsh$ is usually called \emph{harmonic} or \emph{stuffle product}. Let $\h^0$ be the subalgebra of $\h^1$ generated by all words not starting in $z_1$. Due to the usual power series multiplication, the linear map\footnote{By abuse of notation, we use the same symbol for the maps as well as for the objects. From the context, it should always be clear if we are talking about the maps or the objects. %Whenever we have a map $f$ from $\h^1$ into some algebra we write $f(k_1,\dots,k_r) = f(z_{k_1}\dots z_{k_r})$ for $k_1,\dots,k_r\geq 1$.
} defined on the generators by 
\vspace{-0.1cm}
\begin{align}\begin{split} \label{eq:mapzeta}
\zeta: \h^0&\longrightarrow \mz\\ 
z_{k_1}\dots z_{k_r}&\longmapsto \zeta(k_1,\ldots,k_r)
\end{split}
\end{align}
gives an algebra homomorphism from $(\h^0,\ast)$ to $\mz$.
This homomorphism can be extended to a homomorphism $\zeta^\ast:\h^1\to\mz$, so we obtain elements $\zeta^\ast(k_1,\dots,k_r) \in \R$ for all $k_1,\dots,k_r \geq 1$ called the \emph{stuffle regularized multiple zeta values} (see \cite{IKZ}). In the case $k_1 \geq 2$ these coincide with the multiple zeta values \eqref{eq:defmzv} and they are uniquely determined by this property together with $\zeta^\ast(1) = 0$ and the fact that the $\Q$-linear map $\zeta^*: \h^1 \rightarrow \mz$ defined on the generators by $z_{k_1}\dots z_{k_r} \mapsto \zeta^\ast(k_1,\dots,k_r)$ is an algebra homomorphism from $(\h^1,\ast)$ to $\mz$. \\
Next, consider the alphabet given by the two letters $L_{xy}=\{x,y\}$ and write $\h=\Q\langle L_{xy} \rangle$. Define the product $a\diamond b=0$ for all $a,b\in\Q L_{xy}$, then the corresponding quasi-shuffle product $\qsh$ is the shuffle product, denoted by $\shuffle$. Via the  identification $z_k = x^{k-1}y$ we can view $\h^1$ and $\h^0$ as subalgebras of the shuffle algebra $(\h,\shuffle)$, we have $\h^1 = \Q \one + \h y$ and $\h^0 = \Q \one + x \h y$. Due to the iterated integral expression of multiple zeta values, one obtains that the map \eqref{eq:mapzeta} gives an algebra homomorphism from $(\h^0,\shuffle)$ to $\mz$. There is also a unique extension of the map $\zeta$ to an algebra homomorphism $\zeta^\shuffle: (\h^1,\shuffle) \rightarrow \mz$ given by shuffle regularized multiple zeta values and satisfying $\zeta^\shuffle(1)=0$. These two regularizations differ and their difference can be described explicitly (see \cite[Theorem 1]{IKZ}). For example, we have in depth two for all $k_1,k_2\geq 1$
\begin{equation} \label{eqn:doubleshuffle}
\begin{aligned}
\zeta^*(k_1)\zeta^*(k_2)&=\zeta^*(k_1,k_2)+\zeta^*(k_2,k_1)+\zeta^*(k_1+k_2)\\
&=\sum_{j=1}^{k_1+k_2-1}\left(\binom{j-1}{k_1-1}+\binom{j-1}{k_2-1}\right)\zeta^*(j,k_1+k_2-j) + \delta_{k_1+k_2,2} \zeta^*(2) \,,
\end{aligned}
\end{equation}
where $\delta$ denotes the Kronecker delta. We call these equations obtained by comparing products of shuffle- and stuffle-regularized multiple zeta values the \emph{extended double shuffle equations} (see Definition \ref{def:mouldsatisfiesdsh} for a precise definition in terms of generating series).

Beside the multiple zeta values there are other objects satisfying the extended double shuffle equations. In particular, it is known that there exist (non-trivial) \emph{rational\footnote{We call $\Q$-valued solutions in the following also just rational solutions, which should not get confused with solutions given by rational functions.} solutions to the extended double shuffle equations}, i.e. numbers $\beta(k_1,\ldots,k_r)\in\Q$ for $k_1\geq2,\ k_2,\ldots,k_r\geq1$ and corresponding stuffle and shuffle regularized maps $\beta^\ast,\beta^{\shuffle}:\h^1\to \Q$. In this article, we focus on the stuffle regularized objects and thus we write by abuse of notation $\beta=\beta^\ast$. We restrict to rational solutions, which in depth one for even $k\geq 2$ are given by 
\begin{align}\label{eq:betadepthoneexplicit}
    \beta(k) = \frac{\zeta(k)}{(2\pi i)^k} = - \frac{B_k}{2 k!}
\end{align}  and for odd $k\geq 1$ by $\beta(k) = 0$. These rational numbers also appear as the constant terms in the Fourier expansion of the \emph{Eisenstein series}, defined for $k\geq 1$ by 
\begin{align}\label{eq:eisenstein}
	G(k) = \beta(k)+ \frac{1}{(k-1)!}\sum_{d,m\geq 1} d^{k-1} q^{md} \in \Q\llbracket q \rrbracket \,.
\end{align}
For even $k\geq 4$ these are, when viewed as functions in $\tau \in \Ha = \{ \tau \in \C \mid \operatorname{Im}(\tau) > 0 \}$ with $q=e^{2\pi i\tau}$, modular forms of weight $k$ for the full modular group. In our context, they can also be seen as interpolations between $\zeta(k)$ and $\beta(k)$, i.e. the depth one objects of the two solutions of the extended double shuffle equations mentioned above. More precisely, we have $\lim_{q\rightarrow 0} G(k) = \beta(k)$ and $\lim_{q\rightarrow 1}(1-q)^k G(k) = \zeta(k)$, where the latter is a consequence of \eqref{eq:gareqanalogue}.\\

In this paper, we generalize this idea to arbitrary depths and lift a rational solution $\beta$ satisfying \eqref{eq:betadepthoneexplicit}, to objects\footnote{In particular, the $G(k_1,\dots,k_r)$ depend on the choice of the non-unique rational solution $\beta$, i.e. we should write $G_\beta(k_1,\dots,k_r)$. But to keep notations cleaner we omit the $\beta$.} $G(k_1,\dots,k_r)\in  \Q\llbracket q \rrbracket$, which we call \emph{combinatorial multiple Eisenstein series}. In the case $r=1$ they are exactly given by the Eisenstein series \eqref{eq:eisenstein}. 
The combinatorial multiple Eisenstein series interpolate between $\zeta^\ast$ and $\beta$ in arbitrary depths, i.e. we have for $k_1,\dots,k_r \geq 1$ (Proposition \ref{prop:Glimits})
\begin{align} \begin{split}\label{eq:Glimits}
\lim_{q\rightarrow 0} G(k_1,\dots,k_r) &= \beta(k_1,\dots,k_r),\\
  {\lim_{q\rightarrow 1}}^{*}(1-q)^{k_1+\dots+k_r} G(k_1,\dots,k_r) &= \zeta^\ast(k_1,\dots,k_r)\,,
\end{split}
\end{align}
where the $\lim^*$ indicates that we need to do some regularization in the case $k_1 = 1$ (see \eqref{eq:defreglimit}). The construction of the combinatorial multiple Eisenstein series depends on the choice of the rational solution to the extended double shuffle equations $\beta$, though most of their properties are independent of this choice as we have already seen in \eqref{eq:Glimits}. Moreover, the combinatorial multiple Eisenstein series can also be viewed as a map $G: \h^1 \rightarrow \Q\llbracket q \rrbracket$ satisfying for $w,v \in \h^1$ an analogue of the extended double shuffle equations.
For example as an analogue of \eqref{eqn:doubleshuffle} we have for $k_1,k_2\geq 1$ (Proposition \ref{prop:cmesdsh})
\begin{align} \label{eq:edsGdepth2}
G(k_1)G(k_2)&=G(k_1,k_2)+G(k_2,k_1)+G(k_1+k_2)\\
&=\sum_{j=1}^{k_1+k_2-1}\left(\binom{j-1}{k_1-1}+\binom{j-1}{k_2-1}\right)G(j,k_1+k_2-j) +R_G(k_1,k_2)
 \,,
\end{align}
where the $q$-series $R_G(k_1,k_2)$ is given by
\begin{align} \label{eq:extratermdep2}
    R_G(k_1,k_2) = \begin{cases} \frac{(k_1+k_2-3)!}{(k_1-1)!(k_2-1)!}q \frac{d}{dq} G(k_1+k_2-2) & k_1+k_2\geq 3\\
G(2) & k_1+k_2=2 \end{cases}\,.
\end{align}
Observe that we have $\lim_{q\rightarrow 0} R_G(k_1,k_2)=\delta_{k_1+k_2,2}\beta(2)$ and $\lim_{q\rightarrow 1} (1-q)^{k_1+k_2} R_G(k_1,k_2)=\delta_{k_1+k_2,2}\zeta(2)$.
%Here the additional term $R_G$ is given by derivatives of Eisenstein series.
In particular, the formula \eqref{eq:edsGdepth2} gives an explicit expression for $q \frac{d}{dq} G(k)$ in terms of combinatorial double Eisenstein series by choosing $k_2=2$. This actually works for arbitrary depths, for any $w\in \h^1$  we have (Corollary \ref{cor:mapderivandrelation})
\begin{align}\label{eq:termEwz2}
 q \frac{d}{dq} G(w) = G( z_2 \ast w - z_2 \shuffle w)\,.
\end{align}
This is a nice example for the fact that derivatives are an obstacle for the combinatorial multiple Eisenstein series satisfying the extended double shuffle relations. The expression $G(z_2 \ast w - z_2 \shuffle w)$ does not vanish in general, but it is exactly given by a derivative. So in particular, its constant term (and also its limit for $q\rightarrow 1$) indeed vanishes.

% If both words $w,v \in \h^1$ have arbitrary depth, we have to include generalizations of derivatives. 
In order to deal with derivatives and to include them into the algebraic setup, we consider objects depending on double indices. 
More precisely, we introduce \emph{combinatorial bi-multiple Eisenstein series} $G\bi{k_1,\dots,k_r}{d_1,\dots,d_r} \in \Q\llbracket q \rrbracket$ defined for $k_1,\dots,k_r\geq1$ and $d_1,\dots,d_r\geq 0$. The sum $k_1+\dots+k_r+d_1+\dots+d_r$ is called its \emph{weight}. The combinatorial multiple Eisenstein series are given in the special case $$G(k_1,\dots,k_r) = G\bi{k_1,\dots,k_r}{0,\dots,0}.$$ In general one can think of the combinatorial bi-multiple Eisenstein series as some kind of `partial derivatives' of the combinatorial multiple Eisenstein series, since we have (Proposition \ref{prop:derivativecbmes})
\begin{align*}
	q \frac{d}{dq}	G\bi{k_1,\dots,k_r}{d_1,\dots,d_r}   = \sum_{i=1}^r k_i    	G\bi{k_1,\dots,k_i+1,\dots,k_r}{d_1,\dots,d_i+1,\dots,d_r}\,.
\end{align*}
With this the extra term in \eqref{eq:extratermdep2} can be written as $ R_G(k_1,k_2) = \binom{k_1+k_2-2}{k_1-1} G\bi{k_1+k_2-1}{1}$. For example, as an analogue of the double shuffle equation (which also holds for the rational solution $\beta$)
\begin{align*} 
    \zeta(2,1)\zeta(3) &= \zeta(3,2,1)+\zeta(2,3,1)+\zeta(2,1,3)+\zeta(5,1)+\zeta(2,4)\\
    &= 5\zeta(3,2,1)+2\zeta(2,3,1)+\zeta(2,1,3)+2\zeta(3,1,2)+9\zeta(4,1,1)+\zeta(2,2,2)
\end{align*}
the combinatorial multiple Eisenstein series satisfy
\begin{align}\label{eq:depth2times3example}
G(2,1) G(3) \,&=\, +G(3,2,1)+G(2,3,1)+G(2,1,3)+G(5,1)+G(2,4)\\
    \,&=\, 5G(3,2,1)+2G(2,3,1)+G(2,1,3)+2G(3,1,2)+9G(4,1,1)+G(2,2,2) \nonumber \\
    &\,+ 3 G \bi{4,1}{1,0}+ 3 G\bi{3,2}{0,1} + G\bi{2,3}{0,1} \,.\nonumber 
\end{align}
The additional terms $3 G \bi{4,1}{1,0}+ 3 G\bi{3,2}{0,1} + G\bi{2,3}{0,1}$ in \eqref{eq:depth2times3example} vanish for both limits $q\rightarrow 0$ and $q\rightarrow 1$ (after multiplying with $(1-q)^6$). 

In some special cases, the combinatorial bi-multiple Eisenstein series are modular (Proposition \ref{prop:G(kkk)mod}) and the same holds for some linear combinations (Proposition \ref{prop:modularbetazeta}). But in general, the combinatorial bi-multiple Eisenstein series do not satisfy the modularity condition and it is not clear which linear combinations of them do. 

In contrast to the case of multiple zeta values, we do not describe the relations of the combinatorial bi-multiple Eisenstein series in terms of two different product expressions. Instead, we consider a bi-version of the stuffle product and a second family of relations given by the invariance under a certain involution. This involution has a natural origin coming from the theory of partitions (\cite{B1}, \cite{B2},\cite{BI},\cite{Bri}) and it can be described nicely in terms of generating series. Therefore, we work entirely with generating series for the construction of the combinatorial bi-multiple Eisenstein series. For $r\geq 1$ these are denoted by 
\begin{align*}
\gbg\bi{X_1,\ldots,X_r}{Y_1,\ldots, Y_r} = \sum_{\substack{k_1,\dots,k_r \geq 1 \\ d_1,\dots,d_r\geq 0}}G\bi{k_1,\dots,k_r}{d_1,\dots,d_r} X_1^{k_1-1}\cdots X_r^{k_r-1} \frac{Y_1^{d_1}}{d_1!} \cdots \frac{Y_r^{d_r}}{d_r!} \,,
\end{align*}
and in general a collection of such generating series for all $r$ is called a \emph{bimould} (Definition \ref{def:bimould}). To describe the bi-analogue of the stuffle product we consider the alphabet $\LB=\{z^k_d\mid k\geq 1,d\geq0\}$ and define the quasi-shuffle product $\ast=\qsh$ on $\QLB$ by $z^{k_1}_{d_1}\diamond z^{k_2}_{d_2}=z^{k_1+k_2}_{d_1+d_2}$. Then we call the bimould $\gbg$ \emph{symmetril} (Definition \ref{def:bisymmetril}) if the linear map, defined on the generators by
\[z^{k_1}_{d_1}\dots z^{k_r}_{d_r}\mapsto G\bi{k_1,\ldots,k_r}{d_1,\ldots,d_r}\,, \]
gives an algebra homomorphism from $(\QLB,\ast)$ to $\Q\llbracket q\rrbracket$. The  bimould $\gbg$ is called \emph{swap invariant} (Definition \ref{def:swapinvariant}), if it satisfies for all $r\geq 1$ the functional equation 
\begin{align*}
\gbg\bi{X_1,\dots,X_r}{Y_1,\dots,Y_r} = \gbg\bi{Y_1+\dots+Y_r,Y_1+\dots+Y_{r-1},\dots,Y_1+Y_2,Y_1}{X_r, X_{r-1}-X_r,\dots,X_2-X_3,X_1-X_2}\,,
\end{align*}
% a certain functional equation (see \eqref{eq:swap}),
which implies linear relations among combinatorial bi-multiple Eisenstein series in homogeneous weight. The main result of this work is the following.
\begin{theorem*}[Theorem \ref{thm:mainthm}, Proposition \ref{prop:Glimits}]
Let $\beta$ be a $\Q$-valued solution to the extended double shuffle equations, which is in depth one given by \eqref{eq:betadepthone}. Then there exists a $\Q\llbracket q \rrbracket$-valued bimould $\gbg$, which is symmetril, swap invariant and whose coefficients in depth one are the Eisenstein series \eqref{eq:eisenstein}, i.e.
\vspace{-0.3cm}
\begin{align*}
\gbg\bi{X}{0}=\sum_{k\geq 1} G(k)X^{k-1}.
\end{align*}
The coefficients of the bimould $\gbg$ interpolate between (stuffle regularized) multiple zeta values and the given $\Q$-valued solution $\beta$, i.e., they satisfy \eqref{eq:Glimits}.
\end{theorem*}
From this theorem we get that the combinatorial multiple Eisenstein series $G(k_1,\ldots,k_r)$ satisfy the stuffle product formula. By combining symmetrility and swap invariance of $\gbg$, we get that the combinatorial (bi-)multiple Eisenstein series also satisfy an analogue of the shuffle product formula. This is made explicit in depth two in Proposition \ref{prop:cmesdsh}.
 
The construction of the bimould $\gbg$ is inspired by the calculation of the Fourier expansion of the multiple Eisenstein series $\mathbb{G}$ introduced by Gangl-Kaneko-Zagier (\cite{B2},\cite{GKZ}). We recall this calculation (Theorem \ref{thm:mesfourier}) in Section \ref{sec:mes}. The following diagram provides a rough overview of how the building blocks of our constructions (right-hand side) are related to the classical building blocks of multiple Eisenstein series (left-hand side). 
 
 
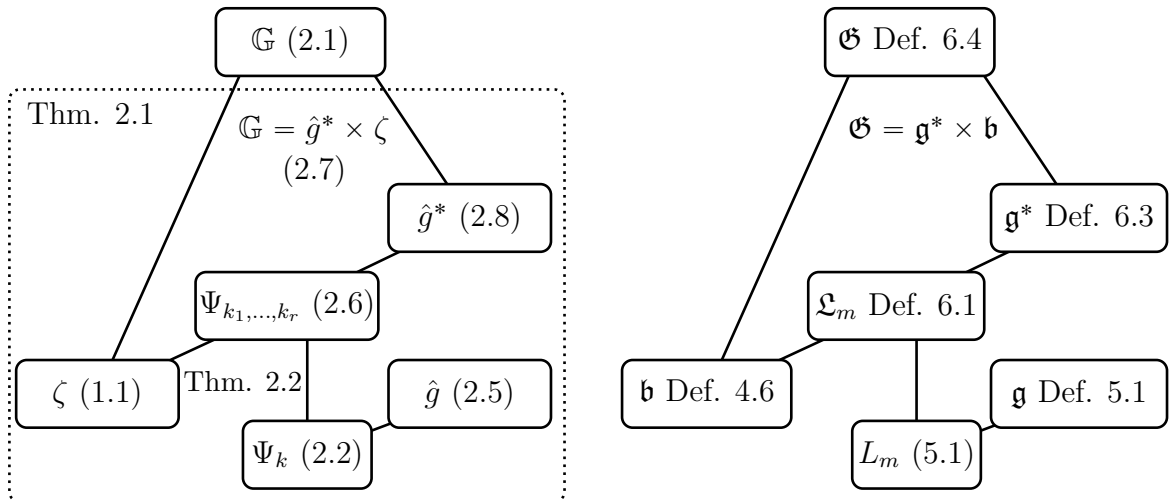
\begin{figure}[!ht]
	\centering
	\tikzset{% \tikzstyle is deprecated
  block/.style={draw, top color=white!90!#1, bottom color=white, shading angle=45, text width=12pt, text centered, rounded corners, minimum height=10pt},
}

\begin{tikzpicture}[scale=0.9, >=Latex, thick] 

\def\linewidth{1pt}
\def\linewidthbig{1pt}
\def\arrowwidth{1pt}

\coordinate (Z) at (0,0);

\def\circrad{0.35}

% 1 MES
\coordinate (PR1) at (3.5,6.6);

% 2 MZV
\coordinate (PR2) at (0.5,1.4);

% 3 g*
\coordinate (PR3) at (6,4);

% 4 Multitangent 
\coordinate (PR4) at (3.3,2.7);

% 11 Monotangent
\coordinate (PR11) at (3.6,0.5);

% 5 g
\coordinate (PR5) at (6,1.4);

% 6 CMES
\coordinate (PR6) at (3.5+9,6.6);

% 7 beta
\coordinate (PR7) at (0.5+9,1.4);

% 8 g*
\coordinate (PR8) at (6+9,4);

% 9 Lm
\coordinate (PR9) at (3.3+9,2.7);

% 12 Lm multi
\coordinate (PR12) at (3.6+9,0.5);

% 10 g
\coordinate (PR10) at (6+9,1.4);

% Connections
\draw[line width=\arrowwidth,->] (3,7) -- (PR2);
\draw[line width=\arrowwidth,->] (4,7) -- (PR3);
\draw[line width=\arrowwidth,-] (PR2) -- (PR4);
\draw[line width=\arrowwidth,-] (PR3) -- (PR4);
\draw[line width=\arrowwidth,-] (3.6,2.5) -- (PR11);
\draw[line width=\arrowwidth,-] (PR5) -- (PR11);

\draw[line width=\arrowwidth,->] (3+9,7) -- (PR7);
\draw[line width=\arrowwidth,->] (4+9,7) -- (PR8);
\draw[line width=\arrowwidth,-] (PR7) -- (PR9);
\draw[line width=\arrowwidth,-] (PR8) -- (PR9);
\draw[line width=\arrowwidth,-] (3.6+9,2.5) -- (PR12);
\draw[line width=\arrowwidth,-] (PR10) -- (PR12);

% 1 MES
\def\titletxt{$\mathbb{G}$ \eqref{eq:defmes}};
\def\n{1}
\coordinate (dim) at (2.5,1);
\coordinate (CC) at (PR\n);
\coordinate (TT) at ($(CC)+0.5*(Z |- dim)+(0,-1)$);
\coordinate (MA) at ($(CC)+(0,-0.5)$);
\coordinate (LL) at ($(CC)-0.5*(dim)$);
\coordinate (UR) at ($(CC)+0.5*(dim)$);
\coordinate (CI) at ($(CC)+(0,0.2)$);

\draw[rounded corners,line width=\linewidthbig,fill=white] (LL) rectangle (UR) {};
\node at (CC) {{\titletxt}};
%\node at (TT) {$\mathbb{G}_{k_1,\dots,k_r}$};

% 2 MZV
\def\titletxt{$\zeta$ \eqref{eq:defmzv}};
\def\n{2}
\coordinate (dim) at (2.4,1);
\coordinate (CC) at (PR\n);
\coordinate (TT) at ($(CC)+0.5*(Z |- dim)+(0,-1)$);
\coordinate (MA) at ($(CC)+(0,-0.5)$);
\coordinate (LL) at ($(CC)-0.5*(dim)$);
\coordinate (UR) at ($(CC)+0.5*(dim)$);
\coordinate (CI) at ($(CC)+(0,0.2)$);

\draw[rounded corners,line width=\linewidthbig,fill=white] (LL) rectangle (UR) {};
\node at (CC) {{\titletxt}};
%\node at (TT) {$\zeta(k_1,\dots,k_r)$};

\node at (3.7,4.7) {{\eqref{eq:classicalmesasmouldproduct}}};
\node at (3.7,5.3) {$\mathbb{G}= \hat{g}^* \times \zeta$};

% 3 g^*
\def\titletxt{$\hat{g}^*$ \eqref{eq:gastclassical}};
\def\n{3}
\coordinate (dim) at (2.4,1);
\coordinate (CC) at (PR\n);
\coordinate (TT) at ($(CC)+0.5*(Z |- dim)+(0,-1)$);
\coordinate (MA) at ($(CC)+(0,-0.5)$);
\coordinate (LL) at ($(CC)-0.5*(dim)$);
\coordinate (UR) at ($(CC)+0.5*(dim)$);
\coordinate (CI) at ($(CC)+(0,0.2)$);

\draw[rounded corners,line width=\linewidthbig,fill=white] (LL) rectangle (UR) {};
\node at (CC) {{\titletxt}};

% 4 Red to monotangent
\def\titletxt{$\Psi_{k_1,\dots,k_r}$ \eqref{eq:defmultitangent}};
\def\n{4}
\coordinate (dim) at (2.7,1);
\coordinate (CC) at (PR\n);
\coordinate (TT) at ($(CC)+0.5*(Z |- dim)+(0,-1)$);
\coordinate (MA) at ($(CC)+(0,-0.5)$);
\coordinate (LL) at ($(CC)-0.5*(dim)$);
\coordinate (UR) at ($(CC)+0.5*(dim)$);
\coordinate (CI) at ($(CC)+(0,0.2)$);

\draw[rounded corners,line width=\linewidthbig,fill=white] (LL) rectangle (UR) {};
\node at (CC) {{\titletxt}};

% 5 g
\def\titletxt{$\hat{g}$ \eqref{eq:ghatclassical}};
\def\n{5}
\coordinate (dim) at (2.4,1);
\coordinate (CC) at (PR\n);
\coordinate (TT) at ($(CC)+0.5*(Z |- dim)+(0,-1)$);
\coordinate (MA) at ($(CC)+(0,-0.5)$);
\coordinate (LL) at ($(CC)-0.5*(dim)$);
\coordinate (UR) at ($(CC)+0.5*(dim)$);
\coordinate (CI) at ($(CC)+(0,0.2)$);

\draw[rounded corners,line width=\linewidthbig,fill=white] (LL) rectangle (UR) {};
\node at (CC) {{\titletxt}};

\draw[rounded corners,line width=\linewidthbig, dotted] (-0.8,5.9) rectangle (7.4,-0.2) {};
\node at (0.4,5.5) {Thm. \mbox{\ref{thm:mesfourier}\strut}};

%
% 11 monotagnet
\def\titletxt{$\Psi_k$ \eqref{eq:defmonotangent}};
\def\n{11}
\coordinate (dim) at (1.9,1);
\coordinate (CC) at (PR\n);
\coordinate (TT) at ($(CC)+0.5*(Z |- dim)+(0,-1)$);
\coordinate (MA) at ($(CC)+(0,-0.5)$);
\coordinate (LL) at ($(CC)-0.5*(dim)$);
\coordinate (UR) at ($(CC)+0.5*(dim)$);
\coordinate (CI) at ($(CC)+(0,0.2)$);

\draw[rounded corners,line width=\linewidthbig,fill=white] (LL) rectangle (UR) {};
\node at (CC) {{\titletxt}};

%
% 12 mono L_m
\def\titletxt{$L_m$ \eqref{eq:deflm}};
\def\n{12}
\coordinate (dim) at (1.9,1);
\coordinate (CC) at (PR\n);
\coordinate (TT) at ($(CC)+0.5*(Z |- dim)+(0,-1)$);
\coordinate (MA) at ($(CC)+(0,-0.5)$);
\coordinate (LL) at ($(CC)-0.5*(dim)$);
\coordinate (UR) at ($(CC)+0.5*(dim)$);
\coordinate (CI) at ($(CC)+(0,0.2)$);

\draw[rounded corners,line width=\linewidthbig,fill=white] (LL) rectangle (UR) {};
\node at (CC) {{\titletxt}};

% 6 CMES
\def\titletxt{$\gbg$ Def. \mbox{\ref{def:defcmes}\strut}};
\def\n{6}
\coordinate (dim) at (2.5,1);
\coordinate (CC) at (PR\n);
\coordinate (TT) at ($(CC)+0.5*(Z |- dim)+(0,-1)$);
\coordinate (MA) at ($(CC)+(0,-0.5)$);
\coordinate (LL) at ($(CC)-0.5*(dim)$);
\coordinate (UR) at ($(CC)+0.5*(dim)$);
\coordinate (CI) at ($(CC)+(0,0.2)$);

\draw[rounded corners,line width=\linewidthbig,fill=white] (LL) rectangle (UR) {};
\node at (CC) {{\titletxt}};

\node at (2.65,1.55) {{\small Thm. \mbox{\ref{thm:reductionmonotangent}\strut}}};

% 7 beta
\def\titletxt{$\bb$ Def. \mbox{\ref{def:b}\strut}};
\def\n{7}
\coordinate (dim) at (2.5,1);
\coordinate (CC) at (PR\n);
\coordinate (TT) at ($(CC)+0.5*(Z |- dim)+(0,-1)$);
\coordinate (MA) at ($(CC)+(0,-0.5)$);
\coordinate (LL) at ($(CC)-0.5*(dim)$);
\coordinate (UR) at ($(CC)+0.5*(dim)$);
\coordinate (CI) at ($(CC)+(0,0.2)$);

\draw[rounded corners,line width=\linewidthbig,fill=white] (LL) rectangle (UR) {};
\node at (CC) {{\titletxt}};
%\node at (TT) {$\zeta(k_1,\dots,k_r)$};

%\node at (3.5,5) {{\eqref{eq:classicalmesasmouldproduct}}};
\node at (3.7+9,5.3) {$\gbg= \gil \times \bb$};

% 8 g^*
\def\titletxt{$\gil$ Def. \ref{def:gil}};
\def\n{8}
\coordinate (dim) at (2.6,1);
\coordinate (CC) at (PR\n);
\coordinate (TT) at ($(CC)+0.5*(Z |- dim)+(0,-1)$);
\coordinate (MA) at ($(CC)+(0,-0.5)$);
\coordinate (LL) at ($(CC)-0.5*(dim)$);
\coordinate (UR) at ($(CC)+0.5*(dim)$);
\coordinate (CI) at ($(CC)+(0,0.2)$);

\draw[rounded corners,line width=\linewidthbig,fill=white] (LL) rectangle (UR) {};
\node at (CC) {{\titletxt}};

% 9 Red to monotangent
\def\titletxt{$\LL_m$ Def. \ref{def:bimouldlm}};
\def\n{9}
\coordinate (dim) at (2.6,1);
\coordinate (CC) at (PR\n);
\coordinate (TT) at ($(CC)+0.5*(Z |- dim)+(0,-1)$);
\coordinate (MA) at ($(CC)+(0,-0.5)$);
\coordinate (LL) at ($(CC)-0.5*(dim)$);
\coordinate (UR) at ($(CC)+0.5*(dim)$);
\coordinate (CI) at ($(CC)+(0,0.2)$);

\draw[rounded corners,line width=\linewidthbig,fill=white] (LL) rectangle (UR) {};
\node at (CC) {{\titletxt}};

% 10 g
\def\titletxt{$\g$ Def. \ref{def:g}};
\def\n{10}
\coordinate (dim) at (2.6,1);
\coordinate (CC) at (PR\n);
\coordinate (TT) at ($(CC)+0.5*(Z |- dim)+(0,-1)$);
\coordinate (MA) at ($(CC)+(0,-0.5)$);
\coordinate (LL) at ($(CC)-0.5*(dim)$);
\coordinate (UR) at ($(CC)+0.5*(dim)$);
\coordinate (CI) at ($(CC)+(0,0.2)$);

\draw[rounded corners,line width=\linewidthbig,fill=white] (LL) rectangle (UR) {};
\node at (CC) {{\titletxt}};

% \draw[step=1cm,gray,very thin] (0,0) grid (16,8);
%  \foreach \cnt in {0,...,15}
%     \node at (\cnt+0.5,0.5) {\cnt};
%  \foreach \cnt in {0,...,7}
%     \node at (0.5,\cnt+0.5) {\cnt};

\end{tikzpicture}
\end{figure}

In particular, the bimould $\gbg$ is constructed out of four bimoulds $\bb$, $\gil$, $\LL_m$ and $\g$, whose constructions are all inspired by the corresponding objects/statements in Section \ref{sec:mes}. We show that the bimoulds $\gil$ and $\bb$ are symmetril (Proposition \ref{prop:gilissymmetril}, \ref{cor:betaswapinvariant}), hence the same holds for the bimould $\gbg$ (by Proposition \ref{prop:productsymmetril}).
On the other hand, the bimould $\gbg$ is a sum of swap invariant bimoulds $\gbg_j$ (Theorem \ref{thm:gjswap}, Proposition \ref{prop:gissumofgj}), thus $\gbg$ is also swap invariant.

By \eqref{eq:Glimits} combinatorial multiple Eisenstein series can also be interpreted as $q$-analogues of multiple zeta values. Our notion of weight is compatible with the weight of quasi-modular forms, and both product expressions of the combinatorial bi-multiple Eisenstein series (given in Proposition \ref{prop:cmesdsh} for depth two) are homogeneous in weight. As far as the authors know, combinatorial multiple Eisenstein series provide the first model of $q$-analogues of multiple zeta values with this property. In particular, this might give a positive answer to a question raised by Okounkov in \cite{O}, since the space $\mathsf{qMZV}$ introduced there is exactly spanned by all $G(k_1,\dots,k_r)$ with $k_1,\dots,k_r\geq 2$. Moreover, we show (Proposition \ref{prop:ginGbi}) that the combinatorial bi-multiple Eisenstein series span the space of $q$-analogues of multiple zeta values $\mz_q$ considered\footnote{In \cite{B1} and \cite{B2} this space is denoted by $\mathcal{BD}$.} in \cite{B1}, \cite{B2}, \cite{BI} and \cite{BK2}.

Conjecturally all algebraic relations among combinatorial bi-multiple Eisenstein are consequences of combining the symmetrility and the swap invariance (Remark \ref{rem:conjgrading}). Since these relations are all in homogeneous weight, we, in particular, expect that the combinatorial bi-multiple Eisenstein are graded by weight.

 In \cite{BI2} the authors introduce the algebra of formal multiple Eisenstein series $\mathcal{G}^f$, which is given by $(\QLB,\ast)$ modulo the relations coming from the swap invariance. In this algebra one can also define a projection to the space of formal multiple zeta values, which can be seen as a formal version of \eqref{eq:Glimits}. Further it is shown, that the $\mathfrak{sl}_2$-action from quasi-modular forms can be extended to this algebra. 
 By the above mentioned conjecture, the algebra of combinatorial bi-multiple Eisenstein series $\CbMES$ (Definition \ref{def:cbmes}) should be isomorphic to the algebra of formal multiple Eisenstein series and therefore $\CbMES$ should also be an $\mathfrak{sl}_2$-algebra.
 
A similar formal algebraic approach is used independently in the thesis of the second named author \cite{Bu}. Here another quasi-shuffle algebra is considered together with an involution, which is of a simpler shape than the operator swap. It is shown in \cite[Theorem 7.10]{Bu2} that this weight-graded algebra is isomorphic to the weight-graded algebra of formal multiple Eisenstein series. The description in terms of this other quasi-shuffle algebra seems to be a good choice to proceed as in \cite{R}, which means giving a generalization of the pro-unipotent affine group scheme $\operatorname{DM}$ and the double shuffle Lie algebra $\mathfrak{dm}_0$.

Finally we remark that the name of the combinatorial multiple Eisenstein series
was inspired by the combinatorial double Eisenstein series $Z_{k_1,k_2}$ introduced in \cite[(17)]{GKZ}. These differ slightly from our $G(k_1,k_2)$, but they can be related using \cite[Proposition 2.5]{BKM} and adding the constant term $\beta(k_1,k_2)$. Combinatorial multiple Eisenstein series might also have a connection to iterated integrals of quasi-modular forms (\cite{M}).
\\

\vspace{0.5cm}

{\bf Acknowledgement:}  We would like to thank the referee for helpful comments and corrections. We also thank Hidekazu Furusho, Jan-Willem van Ittersum, Masanobu Kaneko, Ulf K\"uhn and Wadim Zudilin for fruitful comments on earlier versions of this paper. This project was partially supported by JSPS KAKENHI Grants 19K14499 and 21K13771. The second author was partially funded by the Deutsche Forschungsgemeinschaft (DFG, German Research Foundation) -- SFB-TRR 358/1 2023 — 491392403.

\section{Multiple Eisenstein series}\label{sec:mes}
In this section, we recall multiple Eisenstein series and the calculation of their Fourier expansion. Details can be found in \cite{B1},\cite{B2}, and \cite{BT}. This gives a motivation and an explanation for our construction of combinatorial multiple Eisenstein series in Section \ref{sec:combmes}.

For $k_1 \geq 3, k_2,\dots,k_r \geq 2$ and $\tau \in \Ha$ the \emph{multiple Eisenstein series} are defined by
\begin{align}\label{eq:defmes}
\mathbb{G}_{k_1,\dots,k_r}(\tau) := \sum_{\substack{\lambda_1 \succ \dots \succ \lambda_r \succ 0\\ \lambda_i \in \Z \tau + \Z}} \frac{1}{\lambda_1^{k_1} \dots \lambda_r^{k_r}}   \,,
\end{align}
where the order $\succ$ on the lattice $\Z \tau + \Z$ is defined by $m_1 \tau + n_1 \succ m_2 \tau + n_2$ iff $m_1 > m_2$ or $m_1 = m_2 \wedge n_1 > n_2$. Since $\mathbb{G}_{k_1,\dots,k_r}(\tau + 1) = \mathbb{G}_{k_1,\dots,k_r}(\tau)$ the multiple Eisenstein series possess a Fourier expansion, i.e. an expansion in $q=e^{2\pi i \tau}$, which was calculated in \cite{GKZ} for the $r=2$ case and for arbitrary depth by the first author (\cite{B2}). In depth one we have for $k\geq 3$
\begin{align*} \mathbb{G}_k(\tau) = \sum_{\substack{\lambda \in \Z \tau + \Z\\  \lambda \succ 0}} \frac{1}{\lambda^k }  = \sum_{ \substack{ m > 0 \\ \vee \, (   m=0 \wedge n>0) }} \frac{1}{(m\tau +n)^k }  =  \zeta(k)+  \sum_{m>0} \underbrace{\sum_{n\in \Z}\frac{1}{(m\tau +n)^k}}_{{\large =: \Psi_k(m\tau) } }\,.
\end{align*}

For even $k\geq 4$ these are just the classical Eisenstein series, which are modular forms for the full modular group. When $k$ is even, these differ from the Eisenstein series \eqref{eq:eisenstein} defined in the introduction just by a factor of $(2\pi i)^k$. We refer to $\Psi_k(\tau)$ as the \emph{monotangent function} (\cite{Bo}), which satisfies for $k\geq 2$ the Lipschitz formula
\begin{align}\label{eq:defmonotangent}
    \Psi_k(\tau) = \sum_{n \in \Z} \frac{1}{(\tau+n)^k} = \frac{(-2\pi i)^{k}}{(k-1)!} \sum_{d>0} d^{k-1} q^d  \,.
\end{align}
This gives 
\begin{align*}
\mathbb{G}_k(\tau)  &= \zeta(k) + \sum_{m>0}  \Psi_k(m\tau) =  \zeta(k) + \frac{(-2\pi i)^{k}}{(k-1)!}\sum_{\substack{m>0\\ d >0}} d^{k-1} q^{m d} =: \zeta(k) + (-2\pi i)^k g(k) \,.
\end{align*} 
Here the $g(k)$ are the generating series of the divisor-sums and for higher depths multiple versions of these $q$-series appear, which are defined for $k_1,\dots k_r \geq 1$ by
\begin{align}\label{eq:defmonog}
g(k_1,\dots,k_r)= \sum_{\substack{m_1 > \dots > m_r > 0\\ n_1, \dots , n_r > 0}} \frac{n_1^{k_1-1}}{(k_1-1)!} \dots \frac{n_r^{k_r-1}}{(k_r-1)!}  q^{m_1 n_1 + \dots + m_r n_r } \in \Q\llbracket q \rrbracket\,.
\end{align}
These $q$-series were studied in detail in \cite{B2}, \cite{BK} and they can be seen as $q$-analogues of multiple zeta values since one can show that for $k_1 \geq 2$
\begin{align}\label{eq:gareqanalogue}
\lim\limits_{q\rightarrow 1}(1-q)^{k_1+\dots+k_r }g(k_1,\ldots,k_r) = \zeta(k_1,\dots,k_r)\,.
\end{align}

In the Fourier expansion of (multiple) Eisenstein series, the $q$-series $g$ always appear together with a power of $-2\pi i$ and therefore we set for $k_1,\dots,k_r \geq 1$
\begin{align*}
    \hat{g}(k_1,\dots,k_r) := (-2\pi i)^{k_1+\dots + k_r} g(k_1,\dots,k_r) \in \Q[\pi i]\llbracket  q \rrbracket\,.
\end{align*}
A multiple version of $\mathbb{G}_k(\tau) = \zeta(k) + \hat{g}(k)$ is given by the following. 
\begin{theorem}[$r=1,2$ \cite{GKZ}, $r\geq 1$ \cite{B2}]\label{thm:mesfourier}
 For $k_1 \geq 3, k_2,\dots,k_r \geq 2$ there exist explicit $\alpha^{k_1,\dots,k_r}_{l_1,\dots,l_r,j} \in \Z$, such that for $q=e^{2\pi i \tau}$ we have
\begin{align*}
    \mathbb{G}_{k_1,\dots,k_r}(\tau) = \zeta(k_1,\dots,k_r) +\!\!\!\!\!\!\!\!\!\!\sum_{\substack{0 < j < r\\l_1+\dots+l_r = k_1+\dots+k_r\\l_1\geq 2,l_2,\dots,l_r\geq 1}} \!\!\!\!\!\!\!\!\!\! \alpha^{k_1,\dots,k_r}_{l_1,\dots,l_r,j}\, \zeta(l_1,\dots,l_j)  \hat{g}(l_{j+1},\dots,l_r) +  \hat{g}(k_1,\dots,k_r)\,.
\end{align*}
In particular, $\mathbb{G}_{k_1,\dots,k_r}(\tau) = \zeta(k_1,\dots,k_r)+ \sum_{n> 0} a_{k_1,\dots,k_r}(n) q^n$ for some $a_{k_1,\dots,k_r}(n) \in \mz[\pi i]$.
\end{theorem}
We sketch the proof of Theorem \ref{thm:mesfourier} in the following and then give an explicit example at the end of the section. First, observe that for $k_1,\dots,k_r \geq 2$ we have by the Lipschitz formula \eqref{eq:defmonotangent}, that the $q$-series $\hat{g}$ can be written as an ordered sum over monotangent functions
\begin{align}\label{eq:ghatclassical}
    \hat{g}(k_1,\dots,k_r) = \sum_{m_1 > \dots > m_r > 0} \Psi_{k_1}(m_1 \tau) \cdots \Psi_{k_r}(m_r \tau) \,.
\end{align}
In general the multiple Eisenstein series can be written as ordered sums over \emph{multitangent functions} (\cite{Bo}), which are for  $k_1,\dots,k_r \geq 2$ and $\tau \in \Ha$ defined  by
\begin{align}\label{eq:defmultitangent}
    \Psi_{k_1,\ldots,k_r}(\tau) := \sum_{\substack{n_1>\cdots >n_r \\n_i \in \Z}} \frac{1}{(\tau+n_1)^{k_1}\cdots (\tau+n_r)^{k_r}}.
\end{align}
These functions were originally introduced by Ecalle and then studied in detail by Bouillot in \cite{Bo}. To write $\mathbb{G}_{k_1,\dots,k_r}(\tau)$ in terms of these functions, one splits the summation in the definition \eqref{eq:defmes} into $2^r$ parts, corresponding to the different cases where either $m_i = m_{i+1}$ or $m_i > m_{i+1}$ for $\lambda_i = m_i \tau + n_i$ and $i=1,\dots,{r}$ ($\lambda_{r+1}=0$). Then one can check that the multiple Eisenstein series can be written as 
\begin{align}\label{eq:classicalmesasmouldproduct}
    \mathbb{G}_{k_1,\dots,k_r}(\tau) = \sum_{j=0}^r \hat{g}^*(k_1,\dots,k_j) \zeta(k_{j+1},\dots,k_r)\,,
\end{align}
where the $q$-series $\hat{g}^*$ are given as ordered sums over multitangent functions by
\begin{align}\label{eq:gastclassical}
    \hat{g}^*(k_1,\dots,k_r) := \sum_{\substack{1 \leq j \leq r\\0 = r_0< r_1 < \dots < r_{j-1} < r_j = r\\ m_1 > \dots > m_j > 0}}  \prod_{i=1}^j \Psi_{k_{r_{i-1}+1},\ldots,k_{r_i}}(m_i \tau)\,.
\end{align}
Further, one can show (\cite[Construction 6.7]{B1}) that the $q$-series $\hat{g}^*$ satisfy the harmonic product formula, e.g. $\hat{g}^*(k_1) \hat{g}^*(k_2) = \hat{g}^*(k_1,k_2) + \hat{g}^*(k_2,k_1) + \hat{g}^*(k_1+k_2)$. We will generalize this construction later in terms of generating series (Lemma \ref{lem:bmcm}) and then use an analogue of \eqref{eq:classicalmesasmouldproduct} as the definition for the combinatorial multiple Eisenstein series. To obtain the statement in Theorem \ref{thm:mesfourier} one then uses the following theorem.

% \begin{theorem}{\cite[Theorem 6]{Bo}}\label{thm:reductionmonotangent}
% For $k_1,\dots,k_r \geq 2$ the multitangent function can be written as 
% \vspace{-0.3cm}
% \[ \Psi_{k_1,\dots,k_r}(\tau) = \sum_{j=2}^{k_1+\dots+k_r} \beta^{k_1,\dots,k_r}_{j} \Psi_j (\tau)\]
% with $\beta^{k_1,\dots,k_r}_{j}  \in \mz_{k_1+\dots+k_r-j}$.
% \end{theorem}

\begin{theorem}{\cite[Theorem 6]{Bo}}\label{thm:reductionmonotangent}
For $k_1,\dots,k_r \geq 2$ with $k=k_1+\dots+k_r$ the multitangent function can be written as \\
\scalebox{0.96}{\parbox{.5\linewidth}{% 
\begin{align*}
    \Psi_{k_1,\dots,k_r}(\tau) = \sum_{\substack{1\leq j \leq r\\ l_1+\dots+l_r = k}} (-1)^{l_1+\dots+l_{j-1}+k_j+k} \prod_{\substack{1\leq i \leq r\\i\neq j}}\binom{l_i-1}{k_i-1} \zeta(l_1,\dots,l_{j-1}) \,\Psi_{l_j}(\tau)\, \zeta(l_r,l_{r-1},\dots,l_{j+1})\ .
\end{align*} }} \\
Moreover, the terms containing $\Psi_1(\tau)$ vanish. 
\end{theorem}

This theorem can be proven by using partial fraction decomposition (see Example \ref{ex:mes32}) and then using the shuffle product to show that the coefficient of $\Psi_1(\tau)$ vanishes.

Applying Theorem \ref{thm:reductionmonotangent} to \eqref{eq:gastclassical}, we see by \eqref{eq:ghatclassical}, that the $\hat{g}^*$ can be written as a $\mz$-linear combination of $\hat{g}$. This proves Theorem \ref{thm:mesfourier}, since one can also show that all the appearing multiple zeta values have the correct depth. 

\begin{example}\label{ex:mes32} We give one explicit example in depth two. In this case, \eqref{eq:classicalmesasmouldproduct} reads
\begin{align*}
    \mathbb{G}_{k_1,k_2}(\tau) = \zeta(k_1,k_2) + \hat{g}^\ast(k_1) \zeta(k_2) + \hat{g}^\ast(k_1,k_2)\,,
\end{align*}
where $\hat{g}^\ast(k_1)=\sum_{m_1>0} \Psi_{k_1}(m_1\tau) = \hat{g}(k_1)$ and 
\begin{align*}
\hat{g}^\ast(k_1,k_2) &=  \sum_{m_1>0} \Psi_{k_1,k_2}(m_1 \tau) + \sum_{m_1>m_2>0}  \Psi_{k_1}(m_1 \tau)\Psi_{k_2}(m_2 \tau)\\
&= \sum_{m_1>0} \Psi_{k_1,k_2}(m_1 \tau) + \hat{g}(k_1,k_2)\,.
\end{align*}
Considering the special case $(k_1,k_2)=(3,2)$ one sees by partial fraction decomposition 
{\small
\begin{align*}
\Psi_{3,2}(\tau) &= \sum_{n_1 > n_2} \frac{1}{(\tau+n_1)^3 (\tau+n_2)^2} \\
&= \sum_{n_1 > n_2}  \left( \frac{1}{(n_1-n_2)^2 (\tau+n_1)^3} +\frac{2}{(n_1-n_2)^3 (\tau+n_1)^2} + \frac{3}{(n_1-n_2)^4 (\tau+n_1)} \right) \\
&+\sum_{n_1 > n_2} \left( \frac{1}{(n_1-n_2)^3 (\tau+n_2)^2}  -  \frac{3}{(n_1-n_2)^4 (\tau+n_2)} \right) =  3 \zeta(3) \Psi_2(\tau) +  \zeta(2) \Psi_3(\tau)  \,,
\end{align*}
}and therefore $\hat{g}^\ast(3,2) =  3 \zeta(3) \hat{g}(2) +  \zeta(2) \hat{g}(3) + \hat{g}(3,2)$. In total we get 
\begin{align*}
\mathbb{G}_{3,2}(\tau) =\zeta(3,2) + 3 \zeta(3)  \hat{g}(2) + 2 \zeta(2)  \hat{g}(3) +   \hat{g}(3,2) \,. 
\end{align*}
\end{example}

\section{Moulds, bimoulds and quasi-shuffle products} \label{sec:algsetup}

First, we recall some basic facts on quasi-shuffle products (\cite{Bo2}, \cite{H}, \cite{HI}). Let $L$ be a countable set whose elements we refer to as \emph{letters}. A monic monomial in the non-commutative polynomial ring $\QL$ is called a \emph{word} and we denote the empty word by $\one$. 
Suppose we have a commutative and associative product $\diamond$ on the vector space $\Q L$. Then the \emph{quasi-shuffle product}  $\qsh$ on $\QL$ is defined as the $\Q$-bilinear product, which satisfies $\one \qsh w = w \qsh \one = w$ for any word $w\in \QL$ and
\begin{align}\label{eq:qshdef}
	a w \qsh b v = a (w \qsh b v) + b (a w \qsh v) + (a \diamond b) (w \qsh  v) 
\end{align}
for any letters $a,b \in L$ and words $w, v \in \QL$. This gives a commutative $\Q$-algebra $(\QL, \qsh)$, which is called quasi-shuffle algebra. Moreover, one can equip this algebra with the structure of a Hopf algebra \cite[Section 3]{H}, where the coproduct is given for $w \in \QL$ by the deconcatenation coproduct
\begin{align}\label{eq:coproduct}
    \Delta(w) = \sum_{uv = w} u \otimes v\,.
\end{align}
A well-known example is the shuffle Hopf algebra. Define the product on $\Q L$ by $a\diamond b=0$ for all $a,b \in L$. Then the corresponding quasi-shuffle product $\qsh$ on $\QL$ is the shuffle product, usually denoted by $\shuffle$. The antipode in the shuffle Hopf algebra is given by 
\begin{align*} %\label{eq:antipode} 
S(a_1\dots a_r)=(-1)^r a_r\dots a_1, \qquad a_1,\ldots,a_r\in L,
\end{align*}
so the defining property of $S$ yields the following relations in $\QL$.
\begin{lemma} \label{lem:antipoderel}
For any non-empty word $w=a_1\dots a_r$ in $\QL$, it is
\[\sum_{i=0}^r (-1)^i a_ia_{i-1}\dots a_1 \shuffle a_{i+1}a_{i+2}\dots a_r=0.\]
\end{lemma}
To work with quasi-shuffle products it is convenient to consider generating series. For this, we will introduce the notion of moulds and bimoulds, which were introduced by Ecalle. We refer to the article \cite{Bo2} for a good overview on mould theory and a thorough list of reference for the original works of Ecalle.
\begin{definition}
Let $\A$ be a unital $\Q$-algebra. A family $Z=(Z^{(r)})_{r\geq 0}$ with $Z^{(0)}\in \A$ and $Z^{(r)} \in \A\llbracket X_1,\dots,X_r\rrbracket$ for $r\geq 1$ is called a \emph{mould} with values in $\A$.
\end{definition}
Given a mould $Z=(Z^{(r)})_{r\geq 0}$ we call the $Z^{(r)}$ the depth $r$ part of $Z$. All moulds considered in this article satisfy $Z^{(0)}=1$, so we usually just give the depth $r\geq 1$ parts when defining moulds. Since the depth is clear from the number of variables we just write $Z(X_1,\dots,X_r)$ instead of $Z^{(r)}(X_1,\dots,X_r)$ in the following. Let $Z=(Z^{(r)})_{r \geq 0}$ be an $\A$-valued mould, then we define for $r \geq 1$ and $k_1,\dots,k_r\geq1$ the elements $z(k_1,\dots,k_r) \in \A$ as the  coefficients of its depth $r$ part
\begin{align}\label{eq:defcoefficientz}
Z(X_1,\dots,X_r)=:\sum_{k_1,\dots,k_r\geq 1} z(k_1,\dots,k_r) X_1^{k_1-1} \dots X_r^{k_r-1} \,.
\end{align}
Consider the set of letters $\LZ=\{z_k \mid k\geq 1\}$. Then for any commutative and associative product $\diamond$ on $\Q\LZ$ we obtain a quasi-shuffle algebra $(\QLZ,\qsh)$.
\begin{definition} Let $\A$ be an unital $\Q$-algebra, $Z$ an $\A$-valued mould with coefficients $z$ as defined in \eqref{eq:defcoefficientz}, and $\diamond$ a commutative and associative product on $\Q \LZ$. 
\begin{enumerate}[(i)]
\item The mould $Z$ is called \emph{$\diamond$-symmetril} if the \emph{coefficient map} $\varphi_Z:\QLZ\to \A$ given on the generators by $\varphi_Z(\one)=1$ and 
\[z_{k_1}\dots z_{k_r}\mapsto z(k_1,\ldots,k_r)\]
is an algebra homomorphism from $(\QLZ,\qsh)$ to $\A$.
\item If $\diamond$ is given by $z_{k_1} \diamond z_{k_2} = z_{k_1+k_2}$, then we call a $\diamond$-symmetril mould \emph{symmetril}.
\item If the product $\diamond$ is given by $z_{k_1} \diamond z_{k_2} = 0$, then we call a $\diamond$-symmetril mould \emph{symmetral}.
\end{enumerate}
\end{definition}
Let $Z_1$ and $Z_2$ be moulds with values in $\A$. The \emph{product} $Z_1 \times Z_2$ is the mould given by 
\begin{align*}
(Z_1 \times Z_2)(X_1,\dots,X_r) = \sum_{j=0}^r Z_1(X_1,\dots,X_j)Z_2(X_{j+1},\dots,X_r)\,.
\end{align*}
Since we usually have $Z^{(0)}=1$, the first term ($j=0$) in the above sum is simply given by $Z_2(X_1,\ldots,X_r)$ and the last term ($j=r$) is given by $Z_1(X_1,\ldots,X_r)$. Equipped with this product the space of all ($\A$-valued) moulds becomes a non-commutative $\Q$-algebra.

\begin{definition}\phantomsection \label{def:powseriesmould}
\begin{enumerate}[(i)]
    \item For a mould $Z$ we define the mould $Z^\sharp$ by 
\begin{align*}
    Z^\sharp(X_1,\dots,X_r) &= Z(X_1+\dots+X_r,\dots,X_1+X_2,X_1).
\end{align*}

\item Let $F=\sum_{r\geq 0} a_r T^r \in \A\llbracket T\rrbracket$ be a formal power series with coefficients in $\A$. We can view $F$ as a mould with values in $\A$, which we also denote by $F$ and which is in depth $r\geq 0$ defined by $F^{(r)}(X_1,\dots,X_r)=a_r$. We call such a mould a \emph{constant mould}. \\
Also notice that the product of two constant moulds is exactly the constant mould coming from the product of their power series. 
\item For an $\A$-valued mould $Z$ with coefficients \eqref{eq:defcoefficientz} we define the constant mould $\Gamma^Z$ by 
\begin{align} \label{eq:defgammaZ}
   \Gamma^Z := \sum_{r=0}^\infty \gamma^Z_r T^r := \exp\left( \sum_{n=2}^\infty \frac{(-1)^n}{n} z(n) T^n \right)\,. 
\end{align}
\item For an $\A$-valued mould $Z$ define the mould $$Z_\gamma := Z^\sharp \times \Gamma^Z,$$
i.e. explicitly we have
\begin{align*} 
\Zg(X_1,\dots,X_r) =& \sum_{j=0}^{r} \gamma^Z_{j} Z(X_1+\dots+X_{r-j}, \dots, X_1+X_2,X_1) \,.
\end{align*}
Moreover, we define its coefficients $z_\gamma(k_1,\dots,k_r)\in \A$ by
\begin{align*} 
\Zg(X_1,\dots,X_r)
=:& \sum_{k_1,\dots,k_r\geq 1} z_\gamma(k_1,\dots,k_r) X_1^{k_1-1} \dots X_r^{k_r-1}\,.
\end{align*}
\end{enumerate}
\end{definition}

Conversely, the mould $Z$ can also be written in terms of $\Zg$
\begin{align} \label{eq:Zgammainverse}
    Z(X_1,\dots,X_r) = \sum_{j=0}^r \tilde{\gamma}^Z_j Z_\gamma(X_r, X_{r-1}-X_r,\dots,X_{j+1}-X_{j+2})\,,
\end{align}
where (by using \cite[(32)]{HI} for the last equation) we have
\begin{align}\label{eq:defgammaZinverse}
\sum_{k=0}^\infty \tilde{\gamma}^Z_k T^k := \sum_{k=0}^\infty z(\underbrace{1,\dots,1}_k) T^k  = \exp\left( \sum_{n=2}^\infty \frac{(-1)^{n+1}}{n} z(n) T^n \right)\,.
\end{align}

\begin{definition} \label{def:mouldsatisfiesdsh} Let $\A$ be a unital $\Q$-algebra and $Z$ an $\A$-valued mould. We say that the mould $Z$ \emph{satisfies the extended double shuffle equations} if the mould $Z$ is symmetril and the mould $Z_\gamma$ is symmetral.
\end{definition}

\begin{example}
For a mould $Z$, the extended double shuffle equations in depth two are 
\begin{align}\begin{split}\label{eq:moulddshdep2}
Z(X_1) Z(X_2) &= Z(X_1,X_2) + Z(X_2,X_1) + \frac{Z(X_1)- Z(X_2)}{X_1- X_2}\,,\\
\Zg(X_1) \Zg(X_2) &= \Zg(X_1,X_2)+\Zg(X_2,X_1)\\  &=Z(X_1+X_2,X_1) + Z(X_1+X_2,X_2) + \gamma_2^Z\,.
\end{split}
\end{align}
\end{example}

The motivating example for these equations is the mould of (stuffle regularized) multiple zeta values $\mathfrak{z}$, whose depth $r$ part is defined by 
\begin{align*}
    \mathfrak{z}(X_1,\dots,X_r) = \sum_{k_1,\dots,k_r\geq 1} \zeta^\ast(k_1,\dots,k_r) X_1^{k_1-1} \dots X_r^{k_r-1} \,.
\end{align*}
The mould $\mathfrak{z}$ satisfies the extended double shuffle equations (\cite{IKZ}, \cite{E}) and the corresponding relations obtained for multiple zeta values are exactly the extended double shuffle relations mentioned in the introduction. 

\begin{definition}\label{def:bimould}
Let $\A$ be a unital $\Q$-algebra. A \emph{bimould} with values in $\A$ is a family $B=(B^{(r)})_{r\geq 0}$ with $B^{(0)} \in \A$ and $B^{(r)} \in \A\llbracket X_1,\dots,X_r,Y_1,\dots,Y_r\rrbracket$ for $r\geq 1$.
\end{definition}

As is the case of moulds, we call $B^{(r)}$ the depth $r$ part of $B$ and since the depth is clear from the number of variables and we just write $B\bi{X_1,\dots,X_r}{Y_1,\dots,Y_r}$ instead of $B^{(r)}\bi{X_1,\dots,X_r}{Y_1,\dots,Y_r}$ in the following. Moreover, all appearing bimoulds also satisfy $B^{(0)}=1$, hence we often restrict to the case $r\geq1$ when defining bimoulds.

For bimoulds, we consider the alphabet $$\LB=\{z^k_d \mid k\geq 1, d\geq 0\}\,,$$ which can be seen as a generalization of $L_z$. For a commutative and associative product $\diamond$ on $\Q \LB$ we obtain a quasi-shuffle algebra $(\QLB,\qsh)$.
\begin{definition} \label{def:bisymmetril} Let $\A$ be a unital $\Q$-algebra, $B$ a $\A$-valued bimould, and $\diamond$ a commutative and associative product on $\Q \LB$.
 \begin{enumerate}[(i)]
    \item If the coefficients $b\bi{k_1,\dots,k_r}{d_1,\dots,d_r} \in \A$ of $B$ in depth $r$ are given by
\begin{align*}
     B\bi{X_1,\ldots,X_r}{Y_1,\ldots, Y_r}= \sum_{\substack{k_1,\dots,k_r \geq 1 \\ d_1,\dots,d_r\geq 0}}b\bi{k_1,\dots,k_r}{d_1,\dots,d_r} X_1^{k_1-1}\cdots X_r^{k_r-1} \frac{Y_1^{d_1}}{d_1!} \cdots \frac{Y_r^{d_r}}{d_r!},
\end{align*}
then we define the \emph{coefficient map of $B$} as the $\Q$-linear map $\varphi_B:\QLB \rightarrow \A$ on the generators by $\varphi(\one) = 1$ and 
\begin{align*}
    z^{k_1}_{d_1}\dots z^{k_r}_{d_r} \mapsto b\bi{k_1,\dots,k_r}{d_1,\dots,d_r}\,.
\end{align*}
\item The mould $B$ is called \emph{$\diamond$-symmetril} if the coefficient map is a $\Q$-algebra homomorphism $$\varphi_B: (\QLB,\qsh) \rightarrow \A\,.$$
\item If $\diamond$ is given by $z^{k_1}_{d_1} \diamond z^{k_2}_{d_2} = z^{k_1+k_2}_{d_1+d_2}$, then we call a $\diamond$-symmetril bimould \emph{symmetril}.
\item If the product $\diamond$ is given by $z^{k_1}_{d_1} \diamond z^{k_2}_{d_2} = 0$, then we call a $\diamond$-symmetril bimould \emph{symmetral}.
\end{enumerate}
\end{definition}

If $B$ is symmetril, then in depth two we have as an analogue of the first equation in  \eqref{eq:moulddshdep2}
\begin{align*}
B\bi{X_1}{Y_1}B\bi{X_2}{Y_2} &=B\bi{X_1,X_2}{Y_1,Y_2}+ B\bi{X_2,X_1}{Y_2,Y_1} + \frac{B\bi{X_1}{Y_1+Y_2}- B\bi{X_2}{Y_1+Y_2}}{X_1-X_2} \,.
\end{align*}

Similar as for moulds, we define the \emph{product} of two bimoulds $B$ and $C$ as the bimould $B \times C$ given by 
\begin{align*}
    (B \times C)\bi{X_1,\dots,X_r}{Y_1,\dots,Y_r} = \sum_{j=0}^r B\bi{X_1,\dots,X_j}{Y_1,\dots,Y_j} C\bi{X_{j+1},\dots,X_r}{Y_{j+1},\dots,Y_r}\,.
\end{align*}

\begin{proposition}\label{prop:productsymmetril}
If $B$ and $C$ are $\diamond$-symmetril (bi)moulds then $B \times C$ is $\diamond$-symmetril.
\end{proposition}
\begin{proof} Let $\varphi_B, \varphi_C$ be the coefficient maps of the (bi)moulds $B$ and $C$ and write $m: \A \otimes \A \rightarrow \A$ for the multiplication on $\A$. Then we see by definition that the coefficient map of $B \times C$ is the convolution product of  $\varphi_B$ and $\varphi_C$, i.e.
\begin{align*}
    \varphi_{B \times C} = m \circ (\varphi_B \otimes \varphi_C) \circ \Delta\,,
\end{align*}
where $\Delta$ is the coproduct  \eqref{eq:coproduct} on $(\QL,\qsh)$ for $L=L_z$ or $L=\LB$. This shows that $\varphi_{B \times C}: (\QL,\qsh) \rightarrow \A$ is an algebra homomorphism if $\varphi_B$ and $\varphi_C$ are and therefore $B \times C$ is $\diamond$-symmetril.
\end{proof}
There is another important property of bimoulds, which is closely related to the conjugation of partitions (see \cite{B1}, \cite{B3}, \cite{BI} and \cite{Bri}).
\begin{definition}\label{def:swapinvariant}
A bimould $B$ is called \emph{swap invariant} if for all $r\geq 1$
\begin{align}\label{eq:swap}
B\bi{X_1,\dots,X_r}{Y_1,\dots,Y_r} = B\bi{Y_1+\dots+Y_r,Y_1+\dots+Y_{r-1},\dots,Y_1+Y_2,Y_1}{X_r, X_{r-1}-X_r,\dots,X_2-X_3,X_1-X_2}\,.
\end{align}
\end{definition}
An explicit formula for the coefficients on the right-hand side of \eqref{eq:swap} can be found in \cite[Remark 3.14]{BI}, where the swap is denoted by the involution $\iota$ and the coefficients $b$ are denoted by $P_\mathrm{s}$.

\section{From moulds to bimoulds and the bimould \texorpdfstring{$\bb$}{b}}

Let $\bm$ be a $\Q$-valued mould, which satisfies the extended double shuffle equations and is in depth one given by
\begin{align}\begin{split}\label{eq:betadepthone}
\bm(X) =-\sum_{k\geq 2} \frac{B_k}{2k!}X^{k-1} =  \sum_{m\geq 1}  \frac{\zeta(2m)}{(2\pi i)^{2m}} X^{2m-1}
= \frac{1}{2} \left(\frac{1}{X} - 
\frac{1}{e^X-1}-\frac{1}{2}\right)\,.
\end{split}
\end{align}
%\frac{1}{2X} - \frac{1}{4} \coth\left(\frac{X}{2}\right)
In particular, the coefficients $\beta$ of $\bm$ (as defined in \eqref{eq:defcoefficientz}) are a $\Q$-valued solution to the extended double shuffle equations.
\begin{remark}\label{rem:ratsolexists}
Such a mould $\bb$ satisfying the extended double shuffle equations exists by the work by Racinet \cite{R} or by combining the work of Drinfeld \cite{D} and  Furusho \cite{F}. We give a short explanation how to obtain such an element. In \cite[section IV]{R}, the space $\operatorname{DM}_\lambda(\Q)\subset \Q\langle\langle L_x\rangle \rangle$ is introduced, where $L_x=\{x_0,x_1\}$. It is then shown that the space $\operatorname{DM}_\lambda(\Q)$ is non-empty, so we can choose an element for $\lambda=\beta(2)=-\frac{1}{24}$, i.e. $b\in  \operatorname{DM}_{-\frac{1}{24}}(\Q)$. There is a canonical projection $\pi_z:\Q\langle\langle L_x\rangle\rangle \to \Q\langle\langle L_z\rangle\rangle$, which is given on the generators by $x_0^{k-1}x_1\mapsto z_k$ and maps each word ending in $x_0$ to $0$. So applying the map $z_{k_1}\dots z_{k_r}\mapsto X_1^{k_1-1}\dots X_r^{k_r-1}$ to the depth $r$ component of the element 
\[b_\ast=\exp\left(\sum_{n\geq 2} \frac{(-1)^{n-1}}{n}(\pi_z(b)\mid z_n) z_1^n\right)\pi_z(b) \]
yields a family of generating series $\bm(X_1,\dots,X_r) \in \Q[X_1,\ldots,X_r]$,  which defines a mould $\bb$ satisfying the extended double shuffle equations. \\
In \cite[\S 5]{D}, the space $\operatorname{M}_\mu(\Q)$ of associators is defined for each $\mu\in \Q$. It is shown that the space $\operatorname{M}_\mu(\Q)$ is non-empty, thus we choose an element $b\in \operatorname{M}_1(\Q)$. By \cite[Cor 0.4]{F}, there is an embedding $\operatorname{M}_1(\Q)\hookrightarrow \operatorname{DM}_{-\frac{1}{24}}(\Q)$ (the definition of $\operatorname{DM}_\lambda(\Q)$ in \cite{F} slightly differs from the original one given in \cite{R}, thus one has to be careful with the signs). So take the image of $b$ under the embedding and proceed as before to obtain a mould $\bb$ with values in $\Q$ satisfying the extended double shuffle equations.\\
Neither of the above approaches provides an explicit construction of such a solution, this is an open problem so far. In low depths, there exist explicit rational solutions (\cite{E}, \cite{Br}, \cite{GKZ}), which then give possible candidates for $\bm(X_1,\dots,X_r)$ in the case $r\leq 3$. See also \cite[Section 3]{B3} for an explicit expression of the bimould $\bb$ in depth two coming from the solution presented in \cite{GKZ} or see \cite{BKM} on how to construct directly such a bimould in depth two out of a power series satisfying the Fay-identity (e.g. a variant of $\coth$).
\end{remark}
The mould $\bb$ is not unique, there are different choices starting in weight $8$. In the following, we fix a mould $\bm$ with values in $\Q$, which satisfies the extended double shuffle equations and which is given by \eqref{eq:betadepthone} in depth one. In particular, everything we define in the following depends on this choice.\\

The following gives natural constructions to obtain bimoulds out of moulds. 
% \begin{definition}
% \begin{enumerate}[(i)]
%     \item For a mould $Z$ we define the two bimoulds $\widehat{Z}$ and $\widecheck{Z}$  by 
%     \begin{align*}
%             \widehat{Z} \bi{X_1,\dots,X_r}{Y_1,\dots,Y_r} := Z(X_{1},\dots,X_r),\qquad  \widecheck{Z} \bi{X_1,\dots,X_r}{Y_1,\dots,Y_r} := Z(Y_{1},\dots,Y_r)
%     \end{align*}
% \item For a mould $Z$ we define the bimould $B^Z$ by
% \begin{align}\begin{split} \label{def:bz}
% B^Z = \widecheck{Z_\gamma} \times \widehat{Z}\,,
% \end{split}
% \end{align}
% so explicitly we have \begin{align}\begin{split} \label{def:bz}
% B^Z\bi{X_1,\dots,X_r}{Y_1,\dots,Y_r}  &= \sum_{j=0}^{r}\Zg(Y_1,\dots,Y_{j}) Z(X_{j+1},\dots,X_r) \\
% &=  \sum_{0\leq i\leq  j \leq r} \gamma^Z_i Z(Y_1+\dots+Y_{j-i},\dots,Y_1+Y_2,Y_{1}) Z(X_{j+1},\dots,X_r) \,,
% \end{split}
% \end{align}
% where the coefficients $\gamma^Z_i$ are given by \eqref{eq:defgammaZ}.
% \end{enumerate}
% \end{definition}
%\iffalse Different notation:
\begin{definition}
\begin{enumerate}[(i)]
    \item For a mould $Z$ we define the two bimoulds $X^Z$ and $Y^Z$  by 
    \begin{align*}
            X^Z\bi{X_1,\dots,X_r}{Y_1,\dots,Y_r} := Z(X_{1},\dots,X_r),\qquad Y^Z \bi{X_1,\dots,X_r}{Y_1,\dots,Y_r} := Z(Y_{1},\dots,Y_r)
    \end{align*}
\item For a mould $Z$ we define the bimould $B^Z$ by
\begin{align}\begin{split} \label{def:bz}
B^Z= Y^{Z_\gamma} \times X^Z\,,
\end{split}
\end{align}
so explicitly we have \begin{align}\begin{split} 
B^Z\bi{X_1,\dots,X_r}{Y_1,\dots,Y_r}  &= \sum_{j=0}^{r}\Zg(Y_1,\dots,Y_{j}) Z(X_{j+1},\dots,X_r) \\
&=  \sum_{0\leq i\leq  j \leq r} \gamma^Z_i Z(Y_1+\dots+Y_{j-i},\dots,Y_1+Y_2,Y_{1}) Z(X_{j+1},\dots,X_r) \,,
\end{split}
\end{align}
where the coefficients $\gamma^Z_i$ are given by \eqref{eq:defgammaZ}.
\end{enumerate}
\end{definition}
%\fi

This construction will be used to obtain a bimould version of $\bb$ in Definition \ref{def:b}. We show that $B^Z$ is always a swap invariant bimould and, if $Z$ satisfies the extended double shuffle relations, then $B^Z$ is additionally symmetril. 
\begin{proposition} \label{prop:bzswap}For any mould $Z$ the bimould $B^Z$ is swap invariant.
\end{proposition}
\begin{proof} 
	The swap of $B^Z$ is given by 
	\begin{align*}
	B^Z&\bi{Y_1+\dots+Y_r,Y_1+\dots+Y_{r-1},\dots,Y_1+Y_2,Y_1}{X_r, X_{r-1}-X_r,\dots,X_2-X_3,X_1-X_2}\\
	&= \sum_{0\leq i\leq  j \leq r} \gamma_i Z(X_{r-j+i+1},X_{r-j+i+2},\dots,X_{r}) Z(Y_1+\dots+Y_{r-j},\dots,Y_1) \,.
	\end{align*}
	Making the change of variables $j'=r-j+i$, we see that above sum equals
	\begin{align*}
	\sum_{0\leq i\leq  j' \leq r} \gamma_i Z(X_{j'+1},X_{j'+2},\dots,X_{r}) Z(Y_1+\dots+Y_{j'-i},\dots,Y_1)  = 	B^Z\bi{X_1,\dots,X_r}{Y_1,\dots,Y_r} \,.
	\end{align*}
\end{proof}
\begin{proposition}\label{prop:bzsymmetril} If a mould $Z$ satisfies the extended double shuffle equations, then the bimould $B^Z$ is symmetril.
\end{proposition}
\begin{proof}
Let $Z$ satisfy the extended double shuffle equations, so $Z$ is a symmetril mould and $\Zg$ is a symmetral mould. We immediately obtain that $X^Z$ is a symmetril bimould and $Y^{\Zg}$ is a symmetral bimould. If a bimould does not depend on the variables $X_i$, symmetrility is equivalent to symmetrality. In particular, the bimould $Y^{\Zg}$ is also symmetril. In \eqref{def:bz} we see that $B^Z =Y^{\Zg} \times X^Z$, so by Proposition \ref{prop:productsymmetril} also $B^Z$ is symmetril.
\end{proof}

\begin{remark}
In \cite{BI2}, the relationship between the classical extended double shuffle equations and the relations of the coefficients of swap invariant and symmetril bimoulds will be explained in detail. In particular, the authors show that in the special case of moulds our Definition \ref{def:mouldsatisfiesdsh} coincides with the classical notion of extended double shuffle equations used in \cite{IKZ} and  \cite{R}.
\end{remark}

\begin{definition}\label{def:b} For the fixed $\Q$-valued mould $\bb$ we define its corresponding bimould by $\bb = B^\bb$. By abuse of notation we denote the mould and the bimould by $\bb$, since it becomes clear from the the set of variables which one is meant. Explicitly, we have

\begin{align*}
\bb\bi{X_1,\dots,X_r}{Y_1,\dots,Y_r} = \sum_{0\leq i\leq  j \leq r} \gamma_i  \bm(Y_1+\dots+Y_{j-i},\dots,Y_1+Y_2,Y_{1}) \bm(X_{j+1},\dots,X_r) \,,
\end{align*}
where $\gamma_k=\gamma^\bb_k$ with the notation in \eqref{eq:defgammaZ}, i.e. with \eqref{eq:betadepthone} we have
\begin{align*}
\sum_{k=0}^\infty \gamma_k X^k = \exp\left( \sum_{n=2}^\infty \frac{(-1)^n}{n} \beta(n) X^n \right) = \exp\left( \sum_{n=2}^\infty \frac{(-1)^{n+1}}{n} \frac{B_n}{2 n!} X^n \right) \,.
\end{align*}
\end{definition}

\begin{corollary}\label{cor:betaswapinvariant}
The bimould $\bb$ is swap invariant and symmetril.
\end{corollary}
\begin{proof}
This is just a special case of Propositions \ref{prop:bzswap} and \ref{prop:bzsymmetril}.
\end{proof}

\section{The bimould \texorpdfstring{$\g$}{g}}

For $m\geq 1$, we define the following power series in $\Q\llbracket  q \rrbracket \llbracket X,Y\rrbracket$
\begin{align}\label{eq:deflm}
L_m\bi{X}{Y} = \frac{e^{X+mY}q^m}{1-e^X q^m}\,,
\end{align}
which will be used in the construction of combinatorial multiple Eisenstein series and which is the building block of the following bimould.
\begin{definition}\label{def:g}
We define the bimould $\g$ with values in $\Q\llbracket  q \rrbracket$ by 
\begin{align*}
    \g\bi{X_1,\dots,X_r}{Y_1,\dots,Y_r} = \sum_{m_1 > \dots > m_r > 0} L_{m_1}\bi{X_1}{Y_1} \dots L_{m_r}\bi{X_r}{Y_r}  \,.
\end{align*}
\end{definition}
\begin{proposition}[{\cite[Theorem 2.3]{B1}}]\label{prop:gswapinvariant}
The bimould $\g$ is swap invariant.
\end{proposition}
The coefficients $g$ of $\g$ as defined in Definition \ref{def:bisymmetril} (i) are explicitly given by 
\begin{align}\label{def:big}
    g\bi{k_1,\dots,k_r}{d_1,\dots,d_r} = \sum_{\substack{m_1>\dots>m_r>0\\n_1,\dots,n_r>0}} \frac{n_1^{k_1-1} m_1^{d_1}}{(k_1-1)!} \cdots \frac{n_r^{k_r-1} m_r^{d_r}}{(k_r-1)!} q^{m_1 n_1 + \dots + m_r n_r}\,.
\end{align}
These coefficients are generalizations of the $q$-series defined in \eqref{eq:defmonog}.
The coefficient of $q^n$ is given by the sum over all $m_1 n_1 +\dots+m_r n_r=n$ with $m_1>\dots>m_r>0, n_1,\dots,n_r>0$, i.e. all partitions of $n$ with $r$ different parts $m_1,\dots,m_r$ and multiplicities $n_1,\dots,n_r$.
This sum is invariant under the conjugation of partitions, which on the level of the generating series $\g$ corresponds exactly to the swap invariance \eqref{eq:swap}. This describes a combinatorial proof of Proposition \ref{prop:gswapinvariant}.
Moreover, see \cite{B3} for the interpretation of the coefficients of the bimould $\g$ as a generalization of the generating series of classical divisor-sums and their derivatives.
\vspace{0,3cm} \\
The bimould $\g$ is not symmetril, but we can define a product $\gdiamond$ such that it becomes $\gdiamond$-symmetril. For this, we need the following property of the series $L_m$ defined in \eqref{eq:deflm}.

\begin{lemma} \label{lem:productlm}For all $m\geq 1$ we have
{\small
\begin{align}\label{eq:productlm}
    L_m\bi{X_1}{Y_1}L_m\bi{X_2}{Y_2}& =\frac{L_m\bi{X_1}{Y_1+Y_2} - L_m\bi{X_2}{Y_1+Y_2}}{X_1-X_2} \\
    &+\left(2 \bm(X_2-X_1) - \frac{1}{2}\right) L_m\bi{X_1}{Y_1+Y_2} + \left(2 \bm(X_1-X_2) - \frac{1}{2}\right) L_m\bi{X_2}{Y_1+Y_2}\,,
\nonumber
\end{align}
}where $\bm(X)=-\sum_{k\geq 2} \frac{B_k}{2k!}X^{k-1}$ is the depth one part of the mould $\bm$ defined in \eqref{eq:betadepthone}.
\end{lemma}
\begin{proof}
By direct calculation one checks that $L_m\bi{X}{Y} = \frac{e^{X+mY}q^m}{1-e^X q^m}$ satisfies
\begin{align*}
	L_m\bi{X_1}{Y_1} L_m\bi{X_2}{Y_2} =  \frac{1}{e^{X_1-X_2}-1} L_m\bi{X_1}{Y_1+Y_2} + \frac{1}{e^{X_2-X_1}-1}  L_m\bi{X_2}{Y_1+Y_2},
\end{align*}
which gives the above formula by using $\sum_{n=0}^\infty B_n \frac{X^n}{n!} = \frac{X}{e^X-1}$ and the parity of  $\bm$.
\end{proof}

From this lemma, one can deduce the quasi-shuffle product satisfied by the coefficients $g$ of $\g$. Explicitly, define for $k_1,k_2,j \geq 1$  the rational numbers
\begin{align*}%\label{eq:deflambda}
	\lambda^{k_1,k_2}_j  = -\left((-1)^{k_1} \binom{k_1+k_2-1-j}{k_2 -j} + (-1)^{k_2} \binom{k_1+k_2-1-j}{k_1-j}  \right) \frac{B_{k_1+k_2-j}}{(k_1+k_2-j)!} \,
\end{align*}
and define the commutative and associative product $\gdiamond$ on $\Q\LB$ by 
\begin{align}\label{eq:gdiamond}
	z^{k_1}_{d_1} \gdiamond z^{k_2}_{d_2} = 	z^{k_1+k_2}_{d_1+d_2} +\sum_{j=1}^{k_1+k_2-1} \lambda^{k_1,k_2}_j z^{j}_{d_1+d_2} \,.
\end{align}

\begin{proposition}\label{prop:gisdiamondsymmetril}The bimould $\g$ is $\gdiamond$-symmetril.
\end{proposition}
\begin{proof}
This follows from \cite[Theorem 3.6]{B1} and is a consequence of Lemma \ref{lem:productlm}.  For example, in lowest depth we have 
{\small
\begin{align}\begin{split}\label{eq:gproductindep1dep1}
    \g\bi{X_1}{Y_1}    \g\bi{X_2}{Y_2} &= \left( \sum_{m_1>m_2>0}+ \sum_{m_2>m_1>0}+\sum_{m_1=m_2>0}\right) 	L_{m_1}\bi{X_1}{Y_1} L_{m_2}\bi{X_2}{Y_2} \\
    &\overset{\eqref{eq:productlm}}{=} \g\bi{X_1,X_2}{Y_1,Y_2} + \g\bi{X_2,X_1}{Y_2,Y_1} + \frac{\g\bi{X_1}{Y_1+Y_1} - \g\bi{X_2}{Y_1+Y_1}}{X_1-X_2} \\
    &+ \left(2 \bm(X_2-X_1) - \frac{1}{2}\right) \g\bi{X_1}{Y_1+Y_1} + \left(2 \bm(X_1-X_2) - \frac{1}{2}\right) \g\bi{X_2}{Y_1+Y_1}  \,.
    \end{split}
\end{align} }Considering the coefficients of \eqref{eq:gproductindep1dep1} one sees that $\gdiamond$ in  \eqref{eq:gdiamond} gives $\varphi_\g(z^{k_1}_{d_1} \gqsh z^{k_2}_{d_2}) = \varphi_\g(z^{k_1}_{d_1}) \varphi_\g(z^{k_2}_{d_2})$. The general case can be proven by induction over the depth. 

\end{proof}

Proposition \ref{prop:gisdiamondsymmetril} shows the relationship between the bimould $\g$ and the depth one part of the mould $\bb$. This will play a crucial role in the construction of the combinatorial multiple Eisenstein series. 

\section{Combinatorial multiple Eisenstein series}\label{sec:combmes}
In this section, we introduce combinatorial (bi-)multiple Eisenstein series, which are the coefficients of the bimould $\gbg$. Before we can give the definition of $\gbg$ we need to introduce three other bimoulds $\bR$, $\LL_m$, and $\gil$, which all depend on a fixed choice of a symmetril and swap-invariant bimould $\bb$ given in Definition \ref{def:b}.
\subsection{The bimoulds \texorpdfstring{$\bR$}{tilde(b)}, \texorpdfstring{$\LL_m$}{L}, and \texorpdfstring{$\gil$}{g*}}
Similar as in Definition \ref{def:powseriesmould} (ii) we can view the power series $\exp\!\big(-\frac{T}{2}\big)\in\Q\llbracket  T \rrbracket$ as a constant bimould. Moreover, define for any mould $Z$ the mould $Z^{-}$ for each $r\geq 1$ by
\begin{align*}
Z^{-}(X_1,\ldots,X_r)=Z(-X_1,\ldots,-X_r).  
\end{align*}
With this we define the following analogue of $\bb = B^\bb = Y^{\bb_\gamma} \times X^\bb$.
\begin{definition}\label{def:bimouldlm} Define the bimould $\bR$ as a product of bimoulds
\begin{align} \label{eq:defbR}
\bR=Y^{\bb_\gamma^{-}}\times \exp\!\left(-\frac{T}{2}\right) \times X^{\bb}, \end{align}
i.e. for each $r\geq 1$ we have
\begin{align*} 
\bR\!\bi{X_1,\ldots,X_r}{Y_1,\ldots, Y_r}= \sum_{i=0}^r \frac{(-1)^i}{2^i i!} \bb\bi{X_{i+1},\ldots,X_r}{-Y_1,\ldots, -Y_{r-i}}.
\end{align*}
For $m\geq 1$, let $\LL_m$ be the bimould given in depth $r\geq 1$ by \\
\scalebox{0.90}{\parbox{.5\linewidth}{% 
\begin{align*}
\LL_m\bi{X_1,\ldots, X_r}{Y_1,\ldots, Y_r}=\sum_{j=1}^r \bb\bi{X_1-X_j,\ldots,X_{j-1}-X_j}{Y_1,\ldots,Y_{j-1}}L_m\bi{X_j}{Y_1+\dots+ Y_r} \bR\!\bi{X_r-X_j,\ldots, X_{j+1}-X_j}{Y_r,\ldots,Y_{j+1}}\,.
\end{align*} }}
\end{definition}
Observe that the depth one part of $\LL_m$ is exactly given by the series $L_m$ defined in \eqref{eq:deflm}.

\begin{remark}
(i) The bimould $\LL_m$ can be also defined by using the flexion markers introduced in \cite{E} (cf. \cite[Section 7.5.3]{Bo}).\\
(ii) The definition of $\LL_m$ is inspired by the calculation of the Fourier expansion of multiple Eisenstein series. First observe that the power series $L_m$ can be seen as the generating series of the monotangent functions for $Y=0$. Namely, define the `combinatorial version' of the monotangent function  $\Psi^{\text{comb}}_k(\tau) = \frac{1}{(k-1)!} \sum_{d>0} d^{k-1} q^d$ for $k\geq1$ by using simply the right hand side of the Lipschitz formula \eqref{eq:defmonotangent}. Then we see that
\begin{align*}
    \sum_{k\geq 1} \Psi^{\text{comb}}_k(m \tau) X^{k-1}=   \sum_{k\geq 1}\frac{1}{(k-1)!} \sum_{d>0} d^{k-1} q^{md} X^{k-1} = \sum_{d>0} e^{dX} q^{md} = \frac{e^X q^m}{1-e^X q^m} = L_m\bi{X}{0}\,.
\end{align*}
So in analogy to Theorem \ref{thm:reductionmonotangent}, the $\LL_m$ can be seen as the generating series of the combinatorial version of the multitangent functions. In particular, the trifactorization of the mould of monotangent functions (used to prove Theorem \ref{thm:reductionmonotangent})  in \cite[Theorem 5 \& 6 ]{Bo} is similar to our definition of $\LL_m$. The mould consisting of multiple zeta values in \cite{Bo} is in the definition of $\LL_m$ replaced by the mould $\bb$. Moreover, we omit the constant term for $\Psi_1(\tau)$, this is necessary for our construction of combinatorial multiple Eisenstein series for arbitrary indices (see the discussion before Remark 6.14 in \cite{B1}).
\end{remark}

\begin{definition} \label{def:gil}
We define the bimould $\gil$ in depth $r\geq1$ by 
\begin{align*}
    \gil\bi{X_1,\ldots,X_r}{Y_1,\ldots, Y_r} = \sum_{\substack{1 \leq j \leq r\\0 = r_0< r_1 < \dots < r_{j-1} < r_j = r\\ m_1 > \dots > m_j > 0}}  \prod_{i=1}^j \LL_{m_i} \bi{X_{r_{i-1}+1},\ldots,X_{r_i}}{Y_{r_{i-1}+1},\ldots,Y_{r_i}}\,.
\end{align*}
\end{definition}
The definition of the bimould $\gil$ is inspired by the definition of the classical $g^\ast$ in \eqref{eq:gastclassical}. In particular, we show that the bimould $\gil$ is symmetril (Proposition \ref{prop:gilissymmetril}). 

\subsection{The bimould \texorpdfstring{$\gbg$}{G} and combinatorial (bi-)multiple Eisenstein series}
 
 In analogy to \eqref{eq:classicalmesasmouldproduct} we define the bimould $\gbg$ as the product of $\gil$ and $\bb$.

\begin{definition}\phantomsection\label{def:defcmes}
\begin{enumerate}[(i)]
\item  We define the bimould $\gbg$ by 
\begin{align*}
\gbg = \gil \times \bb\,.
\end{align*}
\item  For $k_1,\dots,k_r \geq 1$ and $d_1,\dots,d_r\geq 0$ we define the \emph{combinatorial bi-multiple Eisenstein series} $G\bi{k_1,\dots,k_r}{d_1,\dots,d_r} \in \Q\llbracket  q \rrbracket$ as the coefficients of the bimould $\gbg$,
\begin{align*}
    \sum_{\substack{k_1,\dots,k_r \geq 1 \\ d_1,\dots,d_r\geq 0}}G\bi{k_1,\dots,k_r}{d_1,\dots,d_r} X_1^{k_1-1}\cdots X_r^{k_r-1} \frac{Y_1^{d_1}}{d_1!} \cdots \frac{Y_r^{d_r}}{d_r!} :=  \gbg\bi{X_1,\ldots,X_r}{Y_1,\ldots, Y_r} \,.
\end{align*}
The number $k_1+\dots+k_r+d_1+\dots+d_r$ is called its \emph{weight} and $r$ its \emph{depth}.
\item In the special case $d_1=\dots=d_r=0$ we define the \emph{combinatorial multiple Eisenstein series} for $k_1,\dots,k_r \geq 1$ by 
\begin{align*}
    G(k_1,\dots,k_r) := G\bi{k_1,\dots,k_r}{0,\dots,0}\,.
\end{align*}
The mould given by their generating series is also denoted by $\gbg$, i.e. 
\begin{align*}
     \gbg(X_1,\ldots,X_r) := \gbg\bi{X_1,\ldots,X_r}{0,\ldots, 0}\,.
\end{align*}
\end{enumerate}
\end{definition}

The main result of this work is the following.
\begin{theorem}\label{thm:mainthm}The bimould $\gbg$ is symmetril and swap invariant.
\end{theorem}
\begin{proof}
    We will show that $\gil$ is symmetril (Proposition \ref{prop:gilissymmetril}). Since $\bb$ is also symmetril, we deduce by Proposition \ref{prop:productsymmetril} that $\gbg = \gil \times \bb$ is symmetril. For swap invariance we will show that $\gbg$ can be written as a sum of swap invariant bimoulds $\gbg_j$ (Proposition \ref{prop:gissumofgj}, Theorem \ref{thm:gjswap}) and therefore $\gbg$ is itself swap invariant.
\end{proof}
Before presenting the necessary results for the proof of Theorem \ref{thm:mainthm}, we give some examples and consequences. 

\begin{example} \phantomsection\label{ex:cmesexamples}
\begin{enumerate}[(i)]
    \item In depth $r=1$ we have $\gil\bi{X}{Y}=\g\bi{X}{Y}$ and therefore $$\gbg\bi{X}{Y}=\bb\bi{X}{Y}+\g\bi{X}{Y}\,.$$
    So the coefficients are for $k\geq 1,d\geq 0$ given by 
\begin{align}\label{eq:cmesindepth1}
	G\bi{k}{d} &= -\delta_{d,0} \frac{B_k}{2 k!} - \delta_{k,1}\frac{B_{d+1}}{2 (d+1)} + \frac{1}{(k-1)!}\sum_{m,n\geq 1} m^d n^{k-1} q^{mn}.
\end{align}
    We see that for $k>d\geq 0$ the combinatorial bi-multiple Eisenstein series  $G\bi{k}{d}$ is essentially the $d$-th derivative of the Eisenstein series $G(k-d)$, since
\begin{align*}
	G\bi{k}{d} 
	&= \frac{(k-d-1)!}{(k-1)!} \left(q \frac{d}{dq}\right)^d G(k-d)\,.
\end{align*}
The swap invariance in depth one just states $\gbg\bi{X}{Y}=\gbg\bi{Y}{X}$. On the level of coefficients this gives $G\bi{k}{d} = \frac{d!}{(k-1)!} G\bi{d+1}{k-1}$, which can also be obtained from  \eqref{eq:cmesindepth1}.
 
\item In depth $r=2$ the bimould $\gbg$ is given by 
\begin{align*}
    \gbg\bi{X_1, X_2}{Y_1,Y_2} =  \gil\bi{X_1,X_2}{Y_1,Y_2} + \gil\bi{X_1}{Y_1} \bb\bi{X_2}{Y_2}+\bb\bi{X_1,X_2}{Y_1,Y_2}
\end{align*}
and from the definition of $\gil$ and $\LL_m$ we get
\begin{align*}
	%\gil\bi{X_1}{Y_1} &= \sum_{m > 0} \LL_m\bi{X_1}{Y_1}\,,\\
	\gil\bi{X_1,X_2}{Y_1,Y_2} &= \sum_{m_1> m_2 >0} \LL_{m_1}\bi{X_1}{Y_1}\LL_{m_2}\bi{X_2}{Y_2} + \sum_{m>0} \LL_m\bi{X_1, X_2}{Y_1, Y_2}\\
	&= \g \bi{X_1,X_2}{Y_1,Y_2} + \bb \bi{X_1-X_2}{Y_1} \g\bi{X_2}{Y_1+Y_2} + \g\bi{X_1}{Y_1+Y_2}\bR\bi{X_2-X_1}{Y_2}\,.
\end{align*}
This gives an explicit expression of $G\bi{k_1,k_2}{d_1,d_2}$ in terms of the $\beta$ and the $q$-series $g$, which we omit here. 
    \item 
    In depth $r=2$ the swap invariance reads $\gbg\bi{X_1,X_2}{Y_1,Y_2}=\gbg\bi{Y_1+Y_2,Y_1}{X_2,X_1-X_2}$, which gives for $k_1,k_2\geq 1, d_1,d_2\geq 0$
    \begin{align}\label{eq:swapdepth2}
        G\bi{k_1,k_2}{d_1,d_2} = \sum_{\substack{0 \leq a \leq d_1\\0 \leq b \leq k_2-1}} \frac{(-1)^b}{a!\, b!} \frac{d_1!}{(k_1-1)!} \frac{(d_2+a)!}{(k_2-1-b)!} \,\,G\bi{d_2+1+a,\,d_1+1-a}{k_2-1-b,\,k_1-1+b}\,.
    \end{align}
    In Proposition \ref{prop:cmesdsh} this formula is used to give an analogue of the shuffle product formula for combinatorial multiple Eisenstein series.
    \item We saw in Example \ref{ex:mes32} that $\mathbb{G}_{3,2}$ is given by 
\begin{align*}
\mathbb{G}_{3,2}(\tau) =\zeta(3,2) + 3 \zeta(3)  \hat{g}(2) + 2 \zeta(2)  \hat{g}(3) +   \hat{g}(3,2) \,. 
\end{align*}
In comparison, we get
\begin{align*}
    G(3,2) = \beta(3,2) + 2 \beta(2) g(3) + g(3,2) = g(3,2)-\frac{1}{12}g(3) \,.
\end{align*}
Notice that $\beta(3)=0$ and therefore this is exactly the same expression after replacing $\zeta$ by the rational numbers $\beta$ and $\hat{g}$ by $g$.
\item The combinatorial multiple Eisenstein series $G(2,1,1)$ is given by 
\begin{align*}
    G(2,1,1) &= \beta(2,1,1)+\frac{1}{6} g(2) -g(2,1)+g(2,1,1)\,.
    %\beta(2,1,1)+\left(\frac{1}{12} +4 \beta(1,1) \right) g(2) -g(2,1)+g(2,1,1) \\
\end{align*}
By duality $\beta(4) = \beta(2,1,1)$, but one can check that $G(4) \neq G(2,1,1)$, i.e. the combinatorial multiple Eisenstein series do not satisfy the duality relations. Another example for this is $G(3) \neq G(2,1)$, since 
\begin{align*}
    \sum_{n>0} \frac{n q^n}{(1-q^n)^2}= q \frac{d}{dq} G(1) = G(3) - G(2,1)\,.
\end{align*}
\end{enumerate}

\end{example}

As an analogue of the double shuffle equations of multiple zeta values in depth two \eqref{eqn:doubleshuffle}, the combinatorial bi-multiple Eisenstein series satisfy the following. 
\begin{proposition}\label{prop:cmesdsh}
For $k_1,k_2\geq 1, d_1,d_2 \geq 0$ we have 
{\small
\begin{align*}\begin{split} 
		G\bi{k_1}{d_1} G\bi{k_2}{d_2}  &=  G\bi{k_1,k_2}{d_1,d_2} +G\bi{k_2,k_1}{d_2,d_1} +G\bi{k_1+k_2}{d_1+d_2} \\
		&=  \sum_{\substack{l_1+l_2=k_1+k_2\\ e_1+e_2=d_1+d_2\\l_1,l_2\geq 1, e_1,e_2\geq 0}} \left(\binom{l_1-1}{k_1-1}\binom{d_1}{e_1}(-1)^{d_1-e_1} +   \binom{l_1-1}{k_2-1}\binom{d_2}{e_1} (-1)^{d_2-e_1}  \right) G\bi{l_1,l_2}{e_1,e_2} \\&+\frac{d_1! d_2!}{(d_1+d_2+1)!}\binom{k_1+k_2-2}{k_1-1}G\bi{k_1+k_2-1}{d_1+d_2+1}\,.
	\end{split}
\end{align*}}

\end{proposition}
\begin{proof} The first equality is just a direct consequence of the symmetrility. To show the second equality, first use the swap invariance to get $	G\bi{k_1}{d_1} G\bi{k_2}{d_2}  = \frac{d_1! d_2!}{(k_1-1)!(k_2-1)!} G\bi{d_1+1}{k_1-1} G\bi{d_2+1}{k_2-1}$ and then evaluate this product by the first equality. Using then again the swap invariance in depth one and depth two \eqref{eq:swapdepth2} yields the result. 
\end{proof}

\begin{remark}
 Proposition \ref{prop:cmesdsh} shows that the combinatorial bi-multiple Eisenstein series in depth $\leq 2$ give a realization of the formal double Eisenstein space introduced in \cite{BKM}. This space is exactly defined by formal symbols satisfying the relation in Proposition  \ref{prop:cmesdsh}.
\end{remark}

One can obtain an analogue for the double shuffle relations in arbitrary depths with the same argument as in the proof of Proposition \ref{prop:cmesdsh}. For example, the equation \eqref{eq:depth2times3example} is obtained in this way.

\begin{example} Evaluating $G(2) G(2,1)$ in the two different ways described above and writing out the Fourier expansion
yields:
\begin{align*}
   0 &= G(2,2,1) + 6 G(3,1,1) - G(2,3) - G(4,1) + 2 G\bi{3,1}{1,0} + G\bi{2,2}{0,1} \\
   &= \beta(2,2,1) + 6 \beta(3,1,1) - \beta(2,3) - \beta(4,1) \\
   &+ \left( 2\beta(2)-\beta(2)^2+12 \beta(1,1)+4 \beta(1,3)+6 \beta(2,2)+12 \beta(3,1) - \frac{1}{6}\right)  q \\
   &+ \left(6 \beta(2)-2 \beta(2)^2+60 \beta(1,1)+8 \beta(1,3)+12 \beta(2,2)+24 \beta(3,1) - 1\right)\,q^2\\
   &+ \left( 4 \beta(2)-2 \beta(2)^2+120 \beta(1,1)+8 \beta(1,3)+12 \beta(2,2)+24 \beta(3,1)- \frac{7}{3} \right) q^3 + O(q^4)\,.
\end{align*}
From this equation we can obtain relations among the coefficients $\beta$ in lower weight without using their explicit expression. We get $\beta(2) = -\frac{1}{24}$, $\beta(1,1) = \frac{1}{48}$ and $2 \beta(1, 3) + 3 \beta(2, 2) + 6 \beta(3, 1) = \frac{1}{1152} = \frac{1}{2} \beta(2)^2$. It might be interesting to understand in general, which relations among the $\beta$ can be obtained from the relations among the $G$.
\end{example}

\begin{definition}\label{def:cbmes}
The $\Q$-vector space spanned by all combinatorial bi-multiple Eisenstein series is defined by
\begin{align*}
    \CbMES = \Q + \big\langle G\bi{k_1,\dots,k_r}{d_1,\dots,d_r} \mid r \geq 1, k_1,\dots,k_r \geq 1, d_1,\dots,d_r\geq 0 \big\rangle_\Q \,,
\end{align*}
and the homogeneous subspace of weight $k\geq 0$ is given by $\CbMES_0=\Q$ and for $k\geq 1$ by
\begin{align*}
    \CbMES_k =  \big\langle G\bi{k_1,\dots,k_r}{d_1,\dots,d_r} \in \CbMES \mid  k_1+\dots+k_r+d_1+\dots+d_r=k \big\rangle_\Q \,.
\end{align*}
The subspace spanned by all combinatorial multiple Eisenstein series is denoted by 
\begin{align*}
    \CMES = \Q + \big\langle G(k_1,\dots,k_r)\mid r \geq 1, k_1,\dots,k_r \geq 1 \big\rangle_\Q \,,
\end{align*}
and we set $\CMES_k = \CMES \cap \CbMES_k$.
\end{definition}

\begin{remark} \label{rem:conjgrading}
We expect that all relations among the combinatorial multiple Eisenstein series come from the swap invariance and symmetrility. In particular, this would imply that the combinatorial bi-multiple Eisenstein series are graded by weight, i.e. we expect $\CbMES \overset{?}{=} \bigoplus_{k\geq 0} \CbMES_k$. 
\end{remark}

\begin{proposition}
Both $\CbMES$ and $\CMES$ are $\Q$-algebras containing the algebra of (quasi-) modular forms with rational coefficients, given by $\Q[G(2),G(4),G(6)]$.
\end{proposition}
\begin{proof}
This follows immediately from the symmetrility of $\gbg$ and  \eqref{eq:cmesindepth1}. It can be also obtained from Proposition \ref{prop:ginGbi} (i) below. 
\end{proof}

\begin{proposition} \label{prop:G(kkk)mod}
For $k\geq 1, d\geq 0$ with $k+d$ even, $G\bi{k,\dots,k}{d,\dots,d}$ is a quasi-modular form.
\end{proposition}

\begin{proof} By a classical result for quasi-shuffle algebras (\cite[(32)]{HI}),the generating series of $G\bi{k,\dots,k}{d,\dots,d}$ can be written as
\begin{align}\label{eq:qshexp}
1+ \sum_{r=1}^{\infty} G\bi{\overbrace{k,\dots,k}^r}{d,\dots,d} T^r = \exp\left( \sum_{r=1}^{\infty} (-1)^{r-1} G\bi{r k}{r d} \frac{T^r}{r} \right)\,.
\end{align}
By \eqref{eq:cmesindepth1}, the $G\bi{r k}{r d}$ are quasi-modular for $k+d$ even. Therefore by \eqref{eq:qshexp}, the $G\bi{k,\dots,k}{d,\dots,d}$ are also quasi-modular forms of weight $r(k+d)$ and depth (in the sense of quasi-modular forms) at most $r\cdot \min(d,k-1)$.
\end{proof}

\begin{example} We give one explicit example for $k=d=1$ and $n=2$:
\begin{align*}
   & G\bi{1,1}{1,1} = \beta(2,2) + 2 \beta(3,1) + \beta(2) g \bi{1}{1} -  \frac{1}{2}g\bi{1}{2} + g\bi{1,1}{1,1}\\
    &= \frac{1}{1152} -\frac{1}{24} g(2) - g(3) + g(2,2) + 2 g(3,1)
    = \frac{1}{2} G\bi{1}{1}^2 - \frac{1}{2}G\bi{2}{2} = \frac{1}{2} G(2)^2 - \frac{1}{2} q \frac{d}{dq}G(2)\,.
\end{align*}
Here the first equality comes from the definition of $  G\bi{1,1}{1,1} $, the second equality follows from using explicit values for $\beta$ ,which are unique up to weight $7$, and the swap invariance of $g$.
The third equality comes from \eqref{eq:qshexp}, but could also be obtained from Proposition \ref{prop:gisdiamondsymmetril}. For the last equation we used the swap invariance of $G$ and \eqref{eq:cmesindepth1}.
\end{example}
For some indices one can also give an explicit formula for the $G$ in terms of the $q$-series $g$, e.g. in the case $k=2, d=0$ one can show that
\begin{align*}
    \sum_{r\geq 0} G(\overbrace{2,\dots, 2}^r) T^{2r+1}= \sum_{r\geq 0} g(\overbrace{2,\dots,2}^r) \left(2 \sin\left(\frac{T}{2}\right)\right)^{2r+1}\,.
\end{align*}
To obtain this formula one shows that the generating series over all $r\geq 1$ of the coefficients of $X_1 \dots X_r$ in $\LL_m\bi{X_1,\dots,X_r}{0,\dots,0}$ has a product expression in terms of the $L_m$. Using the Weierstrass product expression of $\sin$ together with our construction then yields the claim after some calculations. It would be interesting to know if in general there are similar expressions for $G(2k,\dots,2k)$ in analogy to the explicit evaluations of $\zeta(2k,\dots,2k)$ for $k\geq 1$.

By construction the combinatorial bi-multiple Eisenstein series $G$ can be written as rational linear combinations of the $q$-series $g$ defined in \eqref{def:big}. The following Proposition shows that also the converse is true. 
 
\begin{proposition} \label{prop:ginGbi}
For all $k_1,\dots,k_r \geq 1$ and $d_1,\dots,d_r\geq 0$ we have 
\begin{align*}
    g\bi{k_1,\dots,k_r}{d_1,\dots,d_r} \in \CbMES\,.
\end{align*}
In particular, the combinatorial bi-multiple Eisenstein series span the space $\mz_q$ of $q$-analogues of multiple zeta values defined in \cite{BK2}.
\end{proposition}
\begin{proof} In depth one we have by definition $G\bi{k}{d}=g\bi{k}{d}+\mu$ for some $\mu\in \Q$ (see \eqref{eq:cmesindepth1}), thus it is $g\bi{k}{d}\in \CbMES$ for all $k\geq 1, d\geq 0$. Moreover, we obtain immediately from the construction of $\gbg$ that for all $k_1,\ldots,k_r\geq 1$ and $d_1,\ldots,d_r\geq 0$
\begin{align*}
G\bi{k_1,\ldots,k_r}{d_1,\ldots,d_r}=g\bi{k_1,\ldots,k_r}{d_1,\ldots,d_r}+(\text{terms only involving g of smaller depths and weights}).
\end{align*}
So induction on the depth shows that each $q$-series $g$ is a $\Q$-linear combination of combinatorial bi-multiple Eisenstein series. The last statement follows from \cite[Theorem 1]{BK2}, where it is shown that the $q$-series $g$ span the space $\mz_q$, i.e. we get $\CbMES=\mz_q$.
\end{proof}

\begin{remark}
\phantomsection\label{rem:GGbi&dimconj}
\begin{enumerate}[(i)]
\item The similar argument as in Proposition \ref{prop:ginGbi} also shows $g(k_1,\ldots,k_r)\in \CMES$ for all $k_1,\ldots,k_r\geq1$. Also the converse holds, i.e. every combinatorial multiple Eisenstein series is also a $\Q$-linear combination of the single indexed $g$ (this follows from equation \eqref{eq:Gjandg}). This is in contrast to the shuffle (\cite{BT}) and stuffle (\cite{B1}) regularized multiple Eisenstein series, where double indexed $g$ are needed when $k_j=1$ for some $j$.
In particular, we have $\CMES = \mz^\circ_q$ where $\mz^\circ_q$ is defined in \cite{BK2}. As a consequence of this, one would expect $\CMES \overset{?}{=} \CbMES$. This was first conjectured in \cite{B1} (see also \cite[Conjecture 5]{BK2}).
\item As explained in Remark \ref{rem:conjgrading}, we expect that $\CbMES$ is graded by weight. By Proposition \ref{prop:ginGbi}, the dimensions of the homogeneous spaces $\CbMES_k$ should coincide with the conjectured dimensions of the weight-graded parts of $\mz_q$ given in \cite[Conjecture 3]{BK2} (and similarly for the associated depth graded parts).
\end{enumerate}
\end{remark}

We explain now why the combinatorial multiple Eisenstein series interpolate between the rational solution $\beta$ and multiple zeta values. 
 In the case $k_1\geq 2, k_2,\dots,k_r\geq 1$ we get as a direct consequence of the proof of Proposition \ref{prop:ginGbi} and \eqref{eq:gareqanalogue} that
\begin{align} \label{eq:admissiblelimit}
\lim\limits_{q\rightarrow 1}(1-q)^{k_1+\dots+k_r } G(k_1,\ldots,k_r) = \zeta(k_1,\dots,k_r)\,.
\end{align}
The limit is independent of the choice of the rational solution to the double shuffle equations $\bb$, since the limit $q\to1$ considers just the highest depth term of the $q$-series $g$ in $G$. In the case $k_1=1$ this limit does not exist, but we can consider a regularized limit, which we describe now. Using the notation as in Section \ref{sec:algsetup} we can, as in the introduction, view the combinatorial multiple Eisenstein series as a $\Q$-linear map defined on the generators by
\begin{align}\begin{split}\label{eq:Gasamap}
  G:   \h^1 &\longrightarrow \CMES\\
   w= z_{k_1}\dots z_{k_r} &\longmapsto G(w) := G(k_1,\dots,k_r)\,.
   \end{split}
\end{align}
Since $\gbg$ is symmetril, $G$ gives an algebra homomorphism from $(\h^1,\ast)$ to $\CMES$. Due to $\h^1 = \h^0[z_1]$ (cf. \cite[Proposition 1]{IKZ}) we can write $w=z_{k_1}\dots z_{k_r} \in \h^1$ for any $k_1,\dots,k_r\geq 1$ uniquely as $w = \sum_{j=0}^r w_j * z_1^{\ast j}$ with $w_j \in \h^0$. Then we define the regularized version of the limit \eqref{eq:admissiblelimit} as
\begin{align}\label{eq:defreglimit}
      {\lim_{q\rightarrow 1}}^*(1-q)^{k_1+\dots+k_r} G(k_1,\dots,k_r) := \lim_{q\rightarrow 1}(1-q)^{k_1+\dots+k_r} G(w_0)=\zeta(w_0)\,.
\end{align}
Notice that if $k_1\geq 2$, then $w=w_0$ and thus \eqref{eq:defreglimit} is equal to \eqref{eq:admissiblelimit}.
\begin{proposition} \label{prop:Glimits} For any $k_1,\dots,k_r\geq 1$ we have 
\begin{align*}
\lim_{q\rightarrow 0} G(k_1,\dots,k_r) &= \beta(k_1,\dots,k_r),\\
  {\lim_{q\rightarrow 1}}^*(1-q)^{k_1+\dots+k_r} G(k_1,\dots,k_r) &= \zeta^\ast(k_1,\dots,k_r)\,.
\end{align*}
\end{proposition}
\begin{proof}
First notice that the constant terms of the combinatorial multiple Eisenstein series are by construction the coefficients in $\bb\bi{X_1,\dots,X_r}{0,\dots,0}$. The bimould $\bb$ was defined by the mould $\bb$ (Definition \ref{def:b}), which  satisfies the extended double shuffle equations. Since the mould $\bb_\gamma$ is symmetral and $\bb_\gamma(0)=0$ (as $\beta(1)=0$), we inductively get $\bb_\gamma(0,\dots,0)=0$. We deduce $\bb\bi{X_1,\dots,X_r}{0,\dots,0} = \bb(X_1,\dots,X_r)$ which shows the first equation. For the second equation observe that the stuffle regularized multiple zeta values $\zeta^\ast$ are essentially defined in the same way as we constructed the regularized limit \eqref{eq:defreglimit}. Then since $G$ and $\zeta^\ast$ are algebra homomorphisms, we obtain the claim from \eqref{eq:admissiblelimit}.
\end{proof}

\begin{remark} \label{rem:CbMESlimit}
In general one can make sense of the limit of $(1-q)^{k_1+\dots+k_r+d_1+\dots+d_r} G\bi{k_1,\dots,k_r}{d_1,\dots,k_r}$ as $q \rightarrow 1$. In \cite{BI} the authors introduce bi-multiple zeta values $\zeta\bi{k_1,\dots,k_r}{d_1,\dots,k_r} \in \mz$ (\cite[Definition 4.18]{BI} after setting $T=0$), which are essentially given by the regularized limit of $(1-q)^{k_1+\dots+k_r+d_1+\dots+d_r} g\bi{k_1,\dots,k_r}{d_1,\dots,k_r}$ as $q\rightarrow 1$ (similar to \eqref{eq:defreglimit}). Using the notion of degree and weight limit introduced in \cite{BI}, one can check (by Proposition \ref{prop:ginGbi}) that all the other terms in $G\bi{k_1,\dots,k_r}{d_1,\dots,k_r}$ have lower degree than $g\bi{k_1,\dots,k_r}{d_1,\dots,k_r}$, so they do not contribute to the weight limit (\cite[Definition 4.3]{BI}). Therefore, the regularized limit of  $(1-q)^{k_1+\dots+k_r+d_1+\dots+d_r}G\bi{k_1,\dots,k_r}{d_1,\dots,k_r}$ as $q \rightarrow 1$ is exactly given by $\zeta\bi{k_1,\dots,k_r}{d_1,\dots,k_r}$. Assuming that $\CbMES = \bigoplus_{k\geq 0} \CbMES_k$ (Remark \ref{rem:conjgrading}) one then would get that the map $G\bi{k_1,\dots,k_r}{d_1,\dots,k_r} \mapsto \zeta\bi{k_1,\dots,k_r}{d_1,\dots,k_r}$ gives
an $\Q$-algebra homomorphism from $\CbMES$ to $\mz$.
\end{remark}

\begin{proposition}\label{prop:modularbetazeta}
If for some $\epsilon_{k_1,\dots,k_r}\in \Q$ with $r\geq 1$ and $k_1+\dots+k_r = k\geq 4$ the $q$-series $$\sum_{\substack{1 \leq r \leq k\\ k_1+\dots+k_r=k}} \epsilon_{k_1,\dots,k_r} G(k_1,\dots,k_r)$$ is a modular form of weight $k$ (after setting $q=e^{2\pi i\tau}$), then we have 
\begin{align*}
   \sum_{\substack{1 \leq r \leq k\\ k_1+\dots+k_r=k}} \epsilon_{k_1,\dots,k_r} \zeta(k_1,\dots,k_r) = (2\pi i)^k \sum_{\substack{1 \leq r \leq k\\ k_1+\dots+k_r=k}} \epsilon_{k_1,\dots,k_r} \beta(k_1,\dots,k_r)\,.
\end{align*}
\end{proposition}
\begin{proof}
If $f(\tau)= a_0 + \sum_{n\geq 1} a_n q^n$ is modular of weight $k$ then $f(-\frac{1}{\tau}) = \tau^k f(\tau)$, i.e.
		\begin{align*}
	\lim_{q\rightarrow 1} (1-q)^k f(q) &= \lim_{\tau\rightarrow 0} ((2\pi i \tau)^k +  O(\tau^{k+1})) f(\tau) = \lim_{\tau\rightarrow 0} (2\pi i)^k f\left(-\frac{1}{\tau}\right) \\&= \lim_{\tau \rightarrow i\infty} (2\pi i)^k f(\tau) =  (2\pi i)^k a_0\,.
\end{align*}
The statement then follows from Proposition \ref{prop:Glimits}. 
\end{proof}
Notice that the converse of Proposition \ref{prop:modularbetazeta} is not true, since $\zeta(2,1,1) = (2\pi i)^4 \beta(2,1,1)$, but, as seen in Example \ref{ex:cmesexamples}, $G(2,1,1)$ is not a multiple of $G(4)$.
\subsection{Symmetrility of \texorpdfstring{$\LL_m$}{L}, \texorpdfstring{$\gil$}{g*} and \texorpdfstring{$\gbg$}{G}  }
In this subsection, we give the proofs for the symmetrility of previously mentioned bimoulds.
\begin{lemma} \label{lem:lmsymmetril}For all $m\geq 1$ the bimould $\LL_m$ is symmetril.
\end{lemma}
\begin{proof}
By replacing $q^m = e^{-T}$ in $L_m\bi{X}{0} = \frac{e^Xq^m}{1-e^X q^m}$, we obtain a new series
\[L_T(X) = \frac{e^{X-T}}{1-e^{X-T}} = 2 \bb(X-T) - \frac{1}{X-T} - \frac{1}{2}.\]
For each $r\geq1$, we define a multiple bi-version of $L_T$ as
{\small
\begin{align*}
    \LL_T\bi{X_1,\dots,X_r}{Y_1,\ldots,Y_r} = \sum_{j=1}^r \bb\bi{X_1-X_j,\dots,X_{j-1}-X_j}{Y_1,\ldots,Y_{j-1}} L_T(X_j)  \bR\bi{X_r-X_j,\dots,X_{j+1}-X_j}{-Y_r,\ldots,-Y_{j+1}}\,.
\end{align*}} Then after the change of variables $q^m = e^{-T}$ and multiplication with $\exp(m(Y_1+\dots+Y_r))$, we obtain precisely the bimould $\LL_m$ .
Moreover, let $\bb_T, \bR_T$ and $M_T$ be the bimoulds given in depth $r\geq 1$ by \\
\scalebox{0.95}{\parbox{.5\linewidth}{% 
\begin{align*} %\label{eq:bTbRTMT}
    \bb_T\bi{X_1,\dots,X_r}{Y_1,\ldots,Y_r} &= \bb\bi{X_1-T,\ldots,X_r-T}{Y_1,\ldots,Y_r}, \qquad
    \bR_T\bi{X_1,\dots,X_r}{Y_1,\ldots,Y_r} = \bR\bi{X_r-T,\ldots,X_1-T}{Y_r,\ldots,Y_1}\,, \nonumber \\
    M_T\bi{X_1,\ldots,X_r}{Y_1,\ldots,Y_r}&=\begin{cases}
    \frac{1}{T-X_1}, & \text{if } r=1 \\ 0, & \text{else}\end{cases}\,.
\end{align*} }} \\
We show that the bimould $\LL_T$ has the following product representation
\begin{align} \label{eq:LTprod} \LL_T = \bb_T \times M_T \times \bR_T. \end{align}
Since all bimoulds on the right hand side of the equation are symmetril, by Proposition \ref{prop:productsymmetril} also $\LL_T$ is a symmetril bimould. Substituting back $e^{-T}=q^m$ and multiplying by $\exp(m(Y_1+\dots+Y_r))$ gives the symmetrility of the bimould $\LL_m$. In depth one, we compute
\begin{align} \begin{split}\label{eq:LTproddep1}
      \bb_T\bi{X}{0} + M_T\bi{X}{0} + \bR_T\bi{X}{0}
      =2\bb(X-T)-\frac{1}{2}+\frac{1}{T-X}=L_T(X).
       %&= \frac{e^{X-T}{1-e^{X-T}} + \frac{1}{X-T} + M_T\bi{X}{Y}\\
       %&=L_T(X)=\LL_T\bi{X}{Y}.
       \end{split}
\end{align}
Substituting \eqref{eq:LTproddep1} in the left hand side of \eqref{eq:LTprod}, we have to show in some given depth $r\geq1$
\begin{align}\begin{split}\label{eq:symformulartoshow}
&\sum_{j=1}^r \bb\bi{X_1-X_j,\ldots,X_{j-1}-X_j}{Y_1,\ldots,Y_{j-1}}\bb_T\bi{X_j}{0}\bR\bi{X_r-X_j,\ldots,X_{j+1}-X_j}{-Y_r,\ldots,-Y_{j+1}} \\
+&\sum_{j=1}^r \bb\bi{X_1-X_j,\ldots,X_{j-1}-X_j}{Y_1,\ldots,Y_{j-1}}\bR_T\bi{X_j}{0}\bR\bi{X_r-X_j,\ldots,X_{j+1}-X_j}{-Y_r,\ldots,-Y_{j+1}} \\
+&\sum_{j=1}^r \frac{1}{T-X_j} 
\bb\bi{X_1-X_j,\ldots,X_{j-1}-X_j}{Y_1,\ldots,Y_{j-1}}\bR\bi{X_r-X_j,\ldots,X_{j+1}-X_j}{-Y_r,\ldots,-Y_{j+1}} \\
-& \sum_{j=0}^r\bb_T\bi{X_1,\ldots,X_j}{Y_1,\ldots,Y_j}\bR_T\bi{X_{j+1},\ldots,X_r}{Y_{j+1},\ldots, Y_r} \\
-&\sum_{j=1}^r \frac{1}{T-X_j} \bb_T\bi{X_1,\ldots,X_{j-1}}{Y_1,\ldots,Y_{j-1}}\bR_T\bi{X_{j+1},\ldots,X_r}{Y_{j+1},\ldots,Y_r} \\
&=0.
\end{split}
\end{align}
Rewrite this equation in terms of the mould $\bb$ by inserting the Definitions \ref{def:b} and \ref{def:bimouldlm}. Then apply symmetrility to the terms $\bb(X_a-X_j,\ldots,X_{j-1}-X_j)$ and $-\bb(T-X_j)$ for $a\in \{1,\ldots,j\}$ in the first row and to the terms $-\bb(T-X_j)$ and $\bb(X_{r-b}-X_j,\ldots,X_{j+1}-X_j)$ for $b\in \{0,\ldots,r-j\}$ in the second row (after making use of the identity $\bb(X)=-\bb(-X)$). Finally, rewrite the equation in terms of the mould $\bb_\gamma$ by using \eqref{eq:Zgammainverse}. Since the mould $\bb_\gamma$ is symmetral, it satisfies by Lemma \ref{lem:antipoderel}
\begin{align} \label{eq:antipoderelbgamma}
\sum_{j=a-1}^{r-b} (-1)^j \bb_\gamma(X_j,X_{j-1},\ldots,X_a)\bb_\gamma(X_{j+1},X_{j+2},\ldots,X_{r-b})=0
\end{align}
for all $1\leq a\leq r-b\leq r$.
Frequently applying the relation \eqref{eq:antipoderelbgamma} proves the above equation except for the terms, where no mould $\bb$ depending on some of the variables $X_1,\ldots,X_r$ appears. 
To show that these terms also vanish, we use an explicit expression for the generating series of  $\tilde{\gamma}_k = \tilde{\gamma}^\bb_k $ defined in \eqref{eq:defgammaZinverse}
\begin{align*}
\tilde{\gamma}(X) = \sum_{k=0}^\infty \tilde{\gamma}^\bb_k X^k := \sum_{k=0}^\infty \beta(\underbrace{1,\dots,1}_k) X^k  = \exp\left( \sum_{n=2}^\infty \frac{(-1)^{n+1}}{n} \beta(n) X^n \right)\,.
\end{align*}
The following expression of the Gamma function (c.f. \cite[(2.1)]{IKZ})
\begin{align*}
    e^{\gamma X} \Gamma(1+X) = \exp\left( \sum_{n=2}^\infty \frac{(-1)^{n}}{n} \zeta(n) X^n \right)\,
\end{align*}
together with the equality $\beta(n) = \frac{\zeta(n)}{(2\pi i)^n}$ for $n$ even, gives 
\begin{align}\label{eq:explicitgammatilde}
    \tilde{\gamma}(X)^2 = \frac{1}{ \Gamma(1+\frac{X}{2\pi i})\Gamma(1-\frac{X}{2\pi i})} = \frac{2}{X} \sinh\left( \frac{X}{2}\right)= \frac{ e^{\frac{X}{2}} - e^{-\frac{X}{2}}}{X}  \,.
\end{align}
Using the definition of $\bR$ as a product of three moulds given in \eqref{eq:defbR} one can write the remaining terms in \eqref{eq:symformulartoshow} as products of moulds for which one can show that they cancel out by using the explicit formula \eqref{eq:explicitgammatilde} together with some straightforward calculation.
\end{proof}

\begin{lemma} \label{lem:bmcm} Let $B_m$ be a family of bimoulds which are $\diamond$-symmetril for all $m\geq 1$. Then the bimould $C_M$ defined by 
\begin{align*}
    C_M\bi{X_1,\ldots,X_r}{Y_1,\ldots, Y_r} = \sum_{\substack{1 \leq j \leq r\\0 = r_0< r_1 < \dots < r_{j-1} < r_j = r\\ M > m_1 > \dots > m_j > 0}}  \prod_{i=1}^j B_{m_i} \bi{X_{r_{i-1}+1},\ldots,X_{r_i}}{Y_{r_{i-1}+1},\ldots,Y_{r_i}}
\end{align*}
is $\diamond$-symmetril for all $M\geq 1$.
\end{lemma}
\begin{proof}
It is $C_1\bi{X_1,\ldots,X_r}{Y_1,\ldots,Y_r}=0$ for $r\geq 1$ and $C_1^{(0)}=1$, thus $C_1$ is a $\diamond$-symmetril bimould. Moreover, one obtains from direct calculations $C_{M+1} = B_M\times C_M$.
Therefore, induction on $M$ and Proposition \ref{prop:productsymmetril} yields the $\diamond$-symmetrility of the bimould $C_{M}$ for all $M\geq 1$.
\end{proof}
\iffalse \begin{align*}
&C_{M+1}\bi{X_1,\ldots,X_r}{Y_1,\ldots,Y_r}=C_M\bi{X_1,\ldots,X_r}{Y_1,\ldots,Y_r}+\sum_{\substack{1\leq j\leq r \\ 0 = r_0< r_1 < \dots < r_{j-1} < r_j = r\\ M = m_1 > m_2 > \dots > m_j > 0}} \prod_{i=1}^j B_{m_i} \bi{X_{r_{i-1}+1},\ldots,X_{r_i}}{Y_{r_{i-1}+1},\ldots,Y_{r_i}} \\
&=C_M\bi{X_1,\ldots,X_r}{Y_1,\ldots,Y_r}+B_M\bi{X_1,\ldots,X_r}{Y_1,\ldots,Y_r}\\
& \hspace{0,4cm} +\sum_{r_1=1}^{r-1} B_M\bi{X_1,\ldots,X_{r_1}}{Y_1,\ldots,Y_{r_1}} \sum_{\substack{2\leq j\leq r \\ r_1 < r_2 < \dots < r_{j-1} < r_j = r\\ M > m_2 > \dots > m_j > 0}} \prod_{i=2}^j B_{m_i} \bi{X_{r_{i-1}+1},\ldots,X_{r_i}}{Y_{r_{i-1}+1},\ldots,Y_{r_i}}\\
&=C_M\bi{X_1,\ldots,X_r}{Y_1,\ldots,Y_r}+B_M\bi{X_1,\ldots,X_r}{Y_1,\ldots,Y_r}+\sum_{r_1=1}^{r-1} B_M\bi{X_1,\ldots,X_{r_1}}{Y_1,\ldots,Y_{r_1}}C_M\bi{X_{r_1+1},\ldots,X_r}{Y_{r_1+1},\ldots,Y_r} \\
&= (B_M\times C_M)\bi{X_1,\ldots,X_r}{Y_1,\ldots,Y_r}\,.
\end{align*} \fi

\begin{proposition} \label{prop:gilissymmetril} The bimould  $\gil$ is symmetril.
\end{proposition}
\begin{proof}
Choosing $B_m = \LL_m$ in Lemma \ref{lem:bmcm} and taking the limit $M\rightarrow \infty$ gives the bimould $\gil$. By Lemma \ref{lem:lmsymmetril} the bimoulds $\LL_m$ are symmetril for all $m\geq 1$, thus we obtain that $\gil$ is symmetril. 
\end{proof}

\begin{remark} The bimould $\gil$ can be seen as variant of the bimould $\g$ which is symmetril instead of $\tilde{\diamond}$-symmetril. It should be remarked that this correction is a completely different to the one obtained by using the maps $\log$ and $\exp$ from \cite{H} and \cite{HI}, which enables one to switch between different quasi-shuffle products over the same alphabet. This other approach is illustrated in \cite[Remark 6.6]{B1}.
\end{remark}

\subsection{Swap invariance}

\begin{lemma}\label{lem:brnegswap}
The bimould $\bR$ satisfies 
\begin{align*}
\bR\bi{X_1,\dots,X_r}{Y_1,\dots,Y_r} = \bR\bi{-Y_1-\dots-Y_r,Y_1-\dots-Y_{r-1},\dots,-Y_1-Y_2,-Y_1}{-X_r, -X_{r-1}+X_r,\dots,-X_2+X_3,-X_1+X_2}\,,
\end{align*}
i.e. it is nearly swap invariant up to some additional signs. 
\end{lemma}
\begin{proof}
Using the swap invariance of $\bb$ (Corollary \ref{cor:betaswapinvariant}) we get
\begin{align*}
    \bR&\bi{Y_1+\dots+Y_r,\dots,Y_1+Y_2,Y_1}{X_r, X_{r-1}-X_r,\dots,X_1-X_2} 
    = \sum_{i=0}^r \frac{(-1)^i}{2^i i!} \bb\bi{Y_1+\dots+Y_{r-i},\ldots,Y_1}{-X_r,\ldots, -X_{i+1} + X_{i+2}} \\
    &=  \sum_{i=0}^r \frac{(-1)^i}{2^i i!} \bb\bi{-X_{i+1},\ldots,-X_r}{Y_1,\ldots, Y_{r-i}} = \bR\bi{-X_1,\dots,-X_r}{-Y_1,\dots,-Y_r}\,.
\end{align*}
\end{proof}

\begin{definition} \label{def:Gj}
For $j\geq 0$ we define the bimould $\gbg_j=(\gbg^{(r)}_j)_{r\geq 0}$ as follows. In the case $j=0$ we set $\gbg_0=\bb$ and in general $\gbg^{(r)}_j=0$ for $j>r$.
If $1\leq j \leq r$ we define
\begin{align*}
    \gbg_j\bi{X_1,\ldots,X_r}{Y_1,\ldots, Y_r} = \sum_{\substack{0 = r_0 < r_1 < \dots < r_j \leq r\\m_1>\dots>m_j>0}} \prod_{i=1}^j \LL_{m_i}\bi{X_{r_{i-1}+1}, \dots, X_{r_i}}{Y_{r_{i-1}+1}, \dots, Y_{r_i}} \bb\bi{X_{{r_j}+1},\dots,X_r}{Y_{{r_j}+1},\dots,Y_r}\,.
\end{align*}
\end{definition}
Notice that we have $\gbg^{(r)}_r=\g^{(r)}$ for any $r\geq 1$, i.e. the $\gbg_j$ can be seen as an interpolation between the swap invariant bimoulds $\bb$ and $\g$. Using the swap invariance of $\bb$ and $\g$ we get the following more general result.

\begin{theorem}\label{thm:gjswap}
The bimould $\gbg_j$ is swap invariant for any $j\geq 0$.
\end{theorem}
\begin{proof} Since $\gbg_0=\bb$ we can assume $1\leq j \leq r$ in the following. For $1\leq a\leq b\leq r$ we use the notation $X_{a}^b:=X_a - X_b$ and $Y_a^b := Y_a + \dots + Y_b$. The $\LL_{m_i}$ can then be written as
{\small
\begin{align*}
\LL_{m_i}\bi{X_{r_{i-1}+1}, \dots, X_{r_i}}{Y_{r_{i-1}+1}, \dots, Y_{r_i}} = \sum_{r_{i-1} < n_i \leq r_i} \bb\bi{X_{r_{i-1}+1}^{n_i},\ldots,X_{n_i-1}^{n_i}}{Y_{r_{i-1}+1},\ldots,Y_{n_i-1}}L_{m_i}\bi{X_{n_i}}{Y_{r_{i-1}+1}^{r_i}} \bR\!\bi{X_{r_i}^{n_i},\ldots, X_{n_i+1}^{n_i}}{Y_{r_i},\ldots,Y_{n_i+1}}\,.
\end{align*}
}By the definition of the bimould $\g$ in \ref{def:g} as a sum over the $L_m$, we therefore obtain 
{\small
\begin{align}\label{eq:gjwithcandg}
    \gbg_j\bi{X_1,\ldots,X_r}{Y_1,\ldots, Y_r} = \sum_{0 = r_0 < n_1 \leq r_1 < \dots 
    < n_j \leq r_j \leq r} C^{r_1,\dots,r_j}_{n_1,\dots,n_j}\bi{X_1,\ldots,X_r}{Y_1,\ldots, Y_r} \g\bi{X_{n_1}, \dots, X_{n_j}}{Y_{1}^{r_1}, \dots, Y_{r_{j-1}+1}^{r_j}} \,,
\end{align}
}where 
{\small
\begin{align}\label{eq:defc}
   \!\!\!\!\!\!\!\!\!\!\!\! C^{r_1,\dots,r_j}_{n_1,\dots,n_j}\bi{X_1,\ldots,X_r}{Y_1,\ldots, Y_r}  = \prod_{i=1}^j \bb\bi{X_{r_{i-1}+1}^{n_i},\ldots,X_{n_i-1}^{n_i}}{Y_{r_{i-1}+1},\ldots,Y_{n_i-1}} \bR\!\bi{X_{r_i}^{n_i},\ldots, X_{n_i+1}^{n_i}}{Y_{r_i},\ldots,Y_{n_i+1}}\bb\bi{X_{{r_j}+1},\dots,X_r}{Y_{{r_j}+1},\dots,Y_r}.
\end{align}
}We want to check that $\gbg_j$ satisfies \eqref{eq:swap}, i.e. that it is invariant under the simultaneous change of variables $X_j \rightarrow Y_1+\dots+Y_{r-j+1}=Y_1^{r-j+1}$ and $Y_j \rightarrow X_{r-j+1}-X_{r-j+2} = X_{r-j+1}^{r-j+2}$ for $j=1,\dots,r$ (here $X_{r+1}:=0$), which imply $X_a^b \rightarrow Y_{r-b+2}^{r-a+1}$ and $Y_a^b \rightarrow X_{r-b+1}^{r-a+2}$. Applying this change of variables we get \\
\scalebox{0.92}{\parbox{.5\linewidth}{% 
\begin{align} \label{eq:Gjandg}
    \gbg_j\bi{Y_1^r,\ldots,Y_1^2,Y_1^1}{X_r^{r+1},X_{r-1}^r,\ldots, X_1^2} = \!\!\!\!\!\!\!\sum_{0 = r_0 < n_1 \leq r_1 < \dots 
    < n_j \leq r_j  \leq r}\!\!\!\!\!\!\!\!\!\!\!\!\!\!\! C^{r_1,\dots,r_j}_{n_1,\dots,n_j}\bi{Y_1^r,\ldots,Y_1^2,Y_1^1}{X_r^{r+1},X_{r-1}^r,\ldots, X_1^2}  \g\bi{Y_1^{r-n_1+1}, \dots, Y_1^{r-n_j+1}}{X_{r-r_1+1}^{r+1}, \dots, X_{r-r_{j}+1}^{r-r_{j-1}+1}} \,.
\end{align} }} \\
Using the swap invariance of $\g$ together with the change of summation variables $n'_i := r - r_{j-i+1}+1$, $r'_i := r - n_{j-i+1} +1$ we see, after noticing that the sum over $0 < n_1 \leq r_1 < \dots  < n_j \leq r_j \leq r$ is the same as the sum over  $0 < n'_1 \leq r'_1 < \dots  < n'_j \leq r'_j  \leq r$, that{\small
\begin{align*}
    \gbg_j\bi{Y_1^r,\ldots,Y_1^2,Y_1^1}{X_r^{r+1},X_{r-1}^r,\ldots, X_1^2} = \!\!\!\!\!\!\!\sum_{0 = r'_0 < n'_1 \leq r'_1 < \dots  < n'_j \leq r'_j  \leq r}\!\!\!\!\!\!\!\!\!\!\!\!\!\!\! C^{r_1,\dots,r_j}_{n_1,\dots,n_j}\bi{Y_1^r,\ldots,Y_1^2,Y_1^1}{X_r^{r+1},X_{r-1}^r,\ldots, X_1^2}  \g\bi{X_{n'_1}, \dots, X_{n'_j}}{Y_{1}^{r'_1}, \dots, Y_{r'_{j-1}+1}^{r'_j}}\,.
\end{align*}}It remains to show that 
\begin{align*}%\label{eq:cswap}
    C^{r_1,\dots,r_j}_{n_1,\dots,n_j}\bi{Y_1^r,\ldots,Y_1^2,Y_1^1}{X_r^{r+1},X_{r-1}^r,\ldots, X_1^2}  = C^{r'_1,\dots,r'_j}_{n'_1,\dots,n'_j}\bi{X_1,\ldots,X_r}{Y_1,\ldots, Y_r}\,,
\end{align*}
but this follows by using the swap invariance of $\bb$ (Corollary \ref{cor:betaswapinvariant}) and the negative swap invariance for $\bR$ (Lemma \ref{lem:brnegswap}) in \eqref{eq:defc} together with reversing the product, i.e. changing $i\rightarrow j-i$. The factor outside the product then becomes the first factor in the product and the former first factor gives the factor outside the product.  
\end{proof}

\begin{proposition}\label{prop:gissumofgj}
For $r\geq 1$ we have 
\begin{align*}
    \gbg\bi{X_1,\ldots,X_r}{Y_1,\ldots, Y_r} = \sum_{j=0}^r  \gbg_j\bi{X_1,\ldots,X_r}{Y_1,\ldots, Y_r}.
\end{align*}
\end{proposition}
\begin{proof}
This follows directly from the definitions of $\gil$, $\gbg$, and $\gbg_j$ (Definitions \ref{def:gil}, \ref{def:defcmes}, \ref{def:Gj}).
\end{proof}

\begin{remark}
In \cite{BT} it was shown that the Fourier expansion of the multiple Eisenstein series $\mathbb{G}$ can be described by using the Goncharov coproduct (\cite{G}). The explicit calculation of this coproduct has strong similarities with \eqref{eq:gjwithcandg}. Also Example \ref{ex:cmesexamples} (iv) shows that there might be connection of our construction to the Goncharov coproduct. In particular, one might expect, in accordance with the results in \cite{BT}, that $G(k_1,\dots,k_r)$ for $k_1,\dots,k_r \geq 2$ is given by the corresponding convolution product of the coefficient maps $\varphi_\bb$ and $\varphi_\g$.  A natural question then is, if the formula \eqref{eq:gjwithcandg} can be interpreted as the depth $j$ part of some convolution product with respect to some coproduct in this `bi-setup', which might be a natural generalization of the Goncharov coproduct.
\end{remark}

\subsection{Derivatives}
Taking the derivative in \eqref{eq:cmesindepth1} gives $	q \frac{d}{dq} 	G\bi{k}{d} = k 	G\bi{k+1}{d+1}$, which is a special case of the following formula in arbitrary depths.

\begin{proposition}\label{prop:derivativecbmes} For $k_1,\dots,k_r \geq 1$ and $d_1,\dots,d_r\geq 0$ we have  
\begin{align}\label{eq:derivative}
	q \frac{d}{dq}	G\bi{k_1,\dots,k_r}{d_1,\dots,d_r}   = \sum_{i=1}^r k_i    	G\bi{k_1,\dots,k_i+1,\dots,k_r}{d_1,\dots,d_i+1,\dots,d_r}\,.
\end{align}
\end{proposition}
\begin{proof} First notice that \eqref{eq:derivative} is equivalent to 
\begin{align}\label{eq:derforgenser}
    q \frac{d}{dq}    \gbg\bi{X_1,\ldots,X_r}{Y_1,\ldots, Y_r} = \sum_{i=1}^r   \frac{\partial}{\partial X_i}\frac{\partial}{\partial Y_i} \gbg\bi{X_1,\ldots,X_r}{Y_1,\ldots, Y_r}\,.
\end{align} 
Since $ q \frac{d}{dq} L_m\bi{X}{Y} = \frac{\partial}{\partial X}\frac{\partial}{\partial Y}L_m\bi{X}{Y}$ \eqref{eq:derforgenser} is also satisfied by $\g$ (see \cite[Proposition 4.2]{B1}). By 
\eqref{eq:gjwithcandg} we then see that \eqref{eq:derforgenser} is already satisfied by the $\gbg_j$ for all $j$, since
{\small 
\begin{align*}
      q \frac{d}{dq}  &\gbg_j\bi{X_1,\ldots,X_r}{Y_1,\ldots, Y_r} = \sum_{0 = r_0 < n_1 \leq r_1 < \dots 
    < n_j \leq r_j \leq r} C^{r_1,\dots,r_j}_{n_1,\dots,n_j}\bi{X_1,\ldots,X_r}{Y_1,\ldots, Y_r} q \frac{d}{dq} \g\bi{X_{n_1}, \dots, X_{n_j}}{Y_{1}^{r_1}, \dots, Y_{r_{j-1}+1}^{r_j}}  \\
    &= \sum_{0 = r_0 < n_1 \leq r_1 < \dots 
    < n_j \leq r_j \leq r} C^{r_1,\dots,r_j}_{n_1,\dots,n_j}\bi{X_1,\ldots,X_r}{Y_1,\ldots, Y_r} \sum_{i=1}^j \frac{\partial}{\partial X_{n_i}}\frac{\partial}{\partial Y_{n_i}} \g\bi{X_{n_1}, \dots, X_{n_j}}{Y_{1}^{r_1}, \dots, Y_{r_{j-1}+1}^{r_j}} \\
     &= \sum_{i=1}^r   \frac{\partial}{\partial X_i}\frac{\partial}{\partial Y_i}\sum_{0 = r_0 < n_1 \leq r_1 < \dots 
    < n_j \leq r_j \leq r} C^{r_1,\dots,r_j}_{n_1,\dots,n_j}\bi{X_1,\ldots,X_r}{Y_1,\ldots, Y_r}  \g\bi{X_{n_1}, \dots, X_{n_j}}{Y_{1}^{r_1}, \dots, Y_{r_{j-1}+1}^{r_j}} \,.
\end{align*}
}In the last equality we used that (by Definition \ref{def:b}) both $\bb$ and $\bR$, and therefore  $C^{r_1,\dots,r_j}_{n_1,\dots,n_j}$, vanish under any $\frac{\partial}{\partial X_i}\frac{\partial}{\partial Y_i}$
and that the $\g$ terms are independent of $X_{i}$ if $i \not \in \{n_1,\dots,n_j\}$, so they vanish under  $\frac{\partial}{\partial X_i}\frac{\partial}{\partial Y_i}$ in these cases. Since $\gbg$ is the sum of the $\gbg_j$ (Proposition \ref{prop:gissumofgj}), we obtain the formula \eqref{eq:derforgenser}.
\end{proof}

\begin{proposition}\label{prop:cmesderivativeformula}
For $k_1,\dots,k_r \geq 1$ we have
\begin{align}\begin{split}\label{eq:cmesderivativeformula}
	    q \frac{d}{dq}  	G(k_1,\dots,k_r)   = &\,G(2)  G(k_1,\dots,k_r) - \sum_{\substack{1 \leq j \leq r\\a+b = k_j+2}} (a-1)  G(k_1,\dots,k_{j-1},a,b,k_{j+1},\dots,k_r) \\
 & - \sum_{\substack{1 \leq i < j \leq r\\a+b = k_j+1}}  k_i G(k_1,\dots,k_i+1,\dots,k_{j-1},a,b,k_{j+1},\dots,k_r)\\
 &- \sum_{\substack{1 \leq i \leq r}}  k_i G(k_1,\dots,k_i+1,\dots,k_r,1) - G(k_1,\dots,k_r,2)\,.
	\end{split}
\end{align}
In particular, the space $\CMES$ is closed under $q \frac{d}{dq}$.
\end{proposition}
\begin{proof}
Since $\gbg$ is symmetril, we have
	\begin{align*}
		\gbg\bi{X}{Y}  & \gbg\bi{X_1,\dots,X_r}{Y_1,\dots,Y_r} = \sum_{j=0}^r \gbg\bi{X_1,\dots,X_{j},X,X_{j+1},X_r}{Y_1,\dots,Y_{j},Y,Y_{j+1},Y_r} \\
		&+ \sum_{j=1}^r \frac{1}{X-X_j} \left(\gbg\bi{X_1,\dots,X,\dots,X_r}{Y_1,\dots,Y + Y_j,\dots,Y_r}  - \gbg\bi{X_1,\dots,X_j,\dots,X_r}{Y_1,\dots,Y + Y_j,\dots,Y_r}  \right) \,.
	\end{align*}
	Sending all $X_j \rightarrow X$, taking the derivative with respect to $Y$ and setting $Y=0$ gives by Proposition
	\ref{prop:derivativecbmes} \\
	\scalebox{0.97}{\parbox{.5\linewidth}{% 
	\begin{align*}
	q \frac{d}{dq} \gbg\bi{X,\dots,X}{Y_1,\dots,Y_r} = \frac{\partial}{\partial Y} \left( \gbg\bi{X}{Y}   \gbg\bi{X,\dots,X}{Y_1,\dots,Y_r} - \sum_{j=0}^r \gbg\bi{X,\dots,X}{Y_1,\dots,Y_j,Y,Y_{j+1},\dots,Y_r} \right)_{|Y=0}\,. 
	\end{align*} }} \\
	Using the swap invariance and renaming the variables we obtain
	\begin{align*}
			q\frac{d}{dq} \gbg\bi{X_1,\dots,X_r}{Y,0,\dots,0}  &= \frac{\partial}{\partial X} \left(  \gbg\bi{X}{Y}  \gbg\bi{X_1,\dots,X_r}{Y,0,\dots,0} \right)_{|X=0}\\
			&- \frac{\partial}{\partial X} \left( \sum_{j=1}^{r+1} \gbg\bi{X_1+X,\dots,X_{j}+X,X_{j},X_{j+1},\dots,X_r}{Y,0,\dots,0} \right)_{|X=0}\,,
	\end{align*}
	where in the last sum we set $X_{r+1}:=0$.
	The case $Y=0$ yields the result by calculating the coefficients of the right-hand side. 
\end{proof}

With the interpretation of $G$ as an algebra homomorphism from $(\h^1,\ast)$ to $\CMES$ in \eqref{eq:Gasamap}, we can give the following reinterpretation and consequence of Proposition \ref{prop:cmesderivativeformula}.

\begin{corollary}\phantomsection\label{cor:mapderivandrelation}\begin{enumerate}[(i)]
    \item For $w\in \h^1$ the derivative of $G(w)$ is given by
\begin{align*}
    q \frac{d}{dq} G(w) = G( z_2 \ast w - z_2 \shuffle w)\,.
\end{align*}
\item  Let $h:\h^1 \rightarrow \h^1$ be the linear map defined on the generators by 
\begin{align*}
    h: w \longmapsto z_2 \ast w - z_2 \shuffle w\,.
\end{align*}
Then for any $v,w \in \h^1$ we have
\begin{align*}
    h(w \ast v) - h(w) \ast v - w \ast h(v) \in \ker G\,.
\end{align*}
\end{enumerate}
\end{corollary}
\begin{proof}
The equation in (i) is a direct consequence of Proposition\ref{prop:cmesderivativeformula} since the sums on the right-hand side of \eqref{eq:cmesderivativeformula} correspond exactly to those indices which appear in the shuffle product of $z_2 = xy$ with $z_{k_1}\dots z_{k_r}=x^{k_1-1}y \dots x^{k_r-1}y$.
For (ii) we use that $G$ is an algebra homomorphism and $q \frac{d}{dq}$ is a derivation on $\Q\llbracket  q \rrbracket$. By (i) we have $ q \frac{d}{dq} G(w)=G(h(w))$ and therefore  get $G( h(w \ast v) - h(w) \ast v - w \ast h(v)) =0$.
\end{proof}

Notice that the map $h$ is not a derivation on $(\h^1,\ast)$, i.e. the relations we obtain among the combinatorial multiple Eisenstein series from Corollary \ref{cor:mapderivandrelation} (ii) are non-trivial. For example, for $v=w=z_1$ we get $G(4) = 2 G(2,2) - 2 G(3,1)$. This is the first relation, and the only in weight $4$, among combinatorial multiple Eisenstein series, since the $q$-series $g(k_1,\dots,k_r)$ do not satisfy relations in lower weight (see \cite[(1.9)]{BK}). It would be interesting to see if one can describe $G(v \ast w - v \shuffle w)$ for arbitrary $w,v \in \h^1$ explicitly and if this can be used to obtain relations among combinatorial multiple Eisenstein series similar to Corollary \ref{cor:mapderivandrelation} (ii).\\

\begin{remark}
In \cite{F2} the Alekseev-Torossian associator, whose coefficients satisfy the extended double shuffle equations, is computed. It turns out that in depth $1$, it satisfies also the additional conditions given in \eqref{eq:betadepthone} (compare to \cite[Example 4.1]{F2}). In general, the coefficients of the AT associator are not rational. But replacing the rational solution $\mathfrak{b}$ to the extended double shuffle equations by the AT associator gives another family of $q$-series whose generating series are symmetril and swap invariant.
\end{remark}

% {\bf Statements and Declarations:} On behalf of all authors, the corresponding author states that there is no conflict of interest. Data sharing not applicable to this article as no datasets were generated or analysed during the current study

%{\bf Declaration:} Data sharing not applicable to this article as no datasets were generated or analysed during the current study.


\begin{thebibliography}{99}

\bibitem[B1]{B1}  H.~Bachmann:
{\itshape The algebra of bi-brackets and regularized multiple Eisenstein series}, J. Number Theory, {\bf 200} (2019), 260--294.

\bibitem[B2]{B2} H.~Bachmann: 
{\itshape Multiple Eisenstein series and $q$-analogues of multiple zeta values}, Periods in Quantum Field Theory and Arithmetic, Springer Proceedings in Mathematics \& Statistics {\bf 314} (2020), 173--235.

\bibitem[B3]{B3} H.~Bachmann:
{\itshape $q$-analogues of multiple zeta values and the formal double Eisenstein space}, Waseda Number Theory Conference Proceedings 2021, \href{https://arxiv.org/abs/2108.08634}{arXiv:2108.08634}.

\bibitem[BI1]{BI} H.~Bachmann, J.-W.~van-Ittersum:
{\itshape Partitions, multiple zeta values and the $q$-bracket}, 
\newblock Selecta Math. {\bf 30}:3 (2024), 46 pp.


\bibitem[BI2]{BI2} H.~Bachmann, J.-W.~van-Ittersum:
\textit{Formal multiple Eisenstein series and their derivations} (with Appendix by N. Matthes), Adv. Math. {\bf 487} (2026).

\bibitem[BK1]{BK}  H.~Bachmann, U.~K\"uhn:
{\itshape The algebra of generating functions for multiple divisor sums and applications to multiple zeta values}, Ramanujan J. {\bf 40} (2016), 605--648. 

\bibitem[BK2]{BK2}  H.~Bachmann, U.~K\"uhn:
{\itshape A dimension conjecture for q-analogues of multiple zeta values}, Theory and Arithmetic, Springer Proceedings in Mathematics \& Statistics {\bf 314} (2020), 237--258.

\bibitem[BKM]{BKM}  H.~Bachmann, U.~K\"uhn, N.~Matthes:
{\itshape Realizations of the formal double Eisenstein space}, \href{https://arxiv.org/abs/2109.04267}{arXiv:2109.04267}, to appear in Tohoku Math. J.

\bibitem[BT]{BT} H.~Bachmann, K. ~Tasaka:
{\itshape The double shuffle relations for multiple Eisenstein series}, Nagoya Math. Journal {\bf 230} (2017), 1--33.   

\bibitem[Bo]{Bo} O.~Bouillot: 
{\itshape The algebra of multi tangent functions}, J. Algebra {\bf 410} (2014), 148--238.

\bibitem[Bo2]{Bo2} O.~Bouillot: 
{\itshape Mould calculus: from primary to secondary mould symmetries
}, IRMA Lectures in Mathematics and Theoretical Physics {\bf 31}, "Algebraic Combinatorics, Resurgence, Moulds and Applications (CARMA) Vol. 2", 19--84. 

\bibitem[Bri]{Bri} B.~Brindle: {\itshape{A unified approach to qMZVs}}, preprint, 	\href{https://arxiv.org/abs/2111.00051}{arXiv:2111.00051}.

\bibitem[Br]{Br} F.~Brown:
{\itshape Zeta elements in depth 3 and the fundamental Lie algebra of a punctured elliptic curve}, Forum Math. Sigma {\bf 5} (2017), Paper No. e1, 56 pp.

\bibitem[Bu]{Bu} A.~Burmester: {\itshape An algebraic approach to multiple q-zeta values}, PhD Thesis, Universit\"at Hamburg (2023).

\bibitem[Bu2]{Bu2} A.~Burmester: {\itshape Balanced multiple q-zeta values}, Adv. Math. {\bf 439} (2024).


\bibitem[D]{D} V. G.~Drinfeld: {\itshape{On quasi-triangular Quasi-Hopf algebras and a group closely related with $\operatorname{Gal}(\overline{\Q}/\Q)$}}, Leningrad Math. J., {\bf 2}  (1991), 829--860.

\bibitem[E]{E}  J.~Ecalle: {\itshape ARI/GARI, la dimorphie et l'arithmétique des multizêtas : un premier bilan}, Journal de Théorie des Nombres de Bordeaux, Tome {\bf 15} (2003) no. 2, 411--478.

\bibitem[F1]{F} H.~Furusho: {\itshape{Double shuffle relation for associators}}, Annals of Mathematics, {\bf 174} (2011), No. 1, 341--360.

\bibitem[F2]{F2} H.~Furusho: {\itshape{On the coefficients of the Alekseev–Torossian associator}}, J. Algebra {\bf 506} (2018), 364--378.

\bibitem[GKZ]{GKZ} H.~Gangl, M.~Kaneko, D.~Zagier:
{\it Double zeta values and modular forms}, in "Automorphic forms and zeta functions" World Sci. Publ., Hackensack, NJ (2006), 71--106.

\bibitem[G]{G} A. B.~Goncharov:
{\itshape Galois symmetries of fundamental groupoids and noncommutative geometry}, Duke Math. J., {\bf 128}(2) (2005), 209--284.
		
\bibitem[H]{H} M. E.~Hoffman: {\itshape Quasi-shuffle products}, J. Algebraic Combin. {\bf 11} (2000), 49--68.

\bibitem[HI]{HI} M. E.~Hoffman, K.~Ihara: {\itshape Quasi-shuffle products revisited}, J. Algebra {\bf 481} (2017), 293--326.

\bibitem[IKZ]{IKZ} K.~Ihara, M.~Kaneko, D.~Zagier:  
{\itshape Derivation and double shuffle relations for multiple zeta values}, Compositio Math. {\bf 142} (2006), 307--338.

\bibitem[M]{M} N.~Matthes: {\itshape On the algebraic structure of iterated integrals of quasimodular forms}, Algebra \& Number Theory {\bf 11} (2017), no.9, 2113--2130.


\bibitem[O]{O} A.~Okounkov. {\itshape Hilbert schemes and multiple q-zeta values}, Funct. Anal. Appl. {\bf 48} (2014), 138--144.

\bibitem[R]{R} G.~Racinet: {\itshape Series generatrices non commutatives de polyzetas et associateurs de Drinfel'd}, École Normale Supérieure, Ph.D. thesis (2000).

\end{thebibliography}
\end{document}